\documentclass[reqno,12pt,letterpaper]{amsart}
\usepackage{amsmath,amssymb,mathrsfs,graphicx,amsthm, mathtools}
\usepackage[dvipsnames]{xcolor}
\usepackage[colorlinks=true,linkcolor=Red,citecolor=Green]{hyperref}
\usepackage[T1]{fontenc}
\usepackage[utf8]{inputenc}
\usepackage{tikz-cd}

\newtheorem{proposition}{Proposition}[section]
\newtheorem{theorem}[proposition]{Theorem}
\newtheorem{lemma}[proposition]{Lemma}

\theoremstyle{definition}
\newtheorem{definition}[proposition]{Definition}

\theoremstyle{remark}
\newtheorem{remark}[proposition]{Remark}
\theoremstyle{definition}
\newtheorem{question}[proposition]{Question}

\numberwithin{equation}{section}

\newcommand{\N}{\mathbb{N}}
\newcommand{\R}{\mathbb{R}}
\newcommand{\Z}{\mathbb{Z}}
\renewcommand{\div}{\mathrm{div}}

\title[Smooth orbit equivalence rigidity for dissipative geodesic flows]{Smooth orbit equivalence rigidity for dissipative geodesic flows}
\author[J.~Echevarría Cuesta]{Javier {Echevarría Cuesta}}
\address{Department of Pure Mathematics and Mathematical Statistics, University of Cambridge, Cambridge CB3 0WB, UK}
\email{je396@cam.ac.uk}

\begin{document}
\begin{abstract}
Let $M$ be a smooth closed oriented surface. Gaussian thermostats on $M$ correspond to the geodesic flows arising from metric connections, including those with non-zero torsion. These flows may not preserve any absolutely continuous measure. We prove that if two Gaussian  thermostats on $M$ with negative thermostat curvature are related by a smooth orbit equivalence isotopic to the identity, then the two background metrics are conformally equivalent via a smooth diffeomorphism of $M$ isotopic to the identity. We also give a relationship between the thermostat forms themselves. Finally, we prove the same result for Anosov magnetic flows.
\end{abstract}
\maketitle

\section{Introduction}
Let $(M,g)$ be a closed oriented Riemannian surface and let $\lambda\in \mathcal{C}^\infty(SM,\R)$ be a smooth function on the unit tangent bundle $\pi: SM\to M$.  We concern ourselves with the dynamical system governed by the equation 
\begin{equation*}
\nabla_{\dot{\gamma}}\dot{\gamma}=\lambda(\gamma, \dot{\gamma})J\dot{\gamma},
\end{equation*}
where $J:TM\to TM$ is the complex structure on $M$ induced by the orientation.

This equation defines a flow $\varphi_t(\gamma(0),\dot{\gamma}(0))\coloneqq(\gamma(t), \dot{\gamma}(t))$ on $SM$ which reduces to the geodesic flow when $\lambda=0$. The flow models the motion of a particle under the influence of a force orthogonal to the velocity and with magnitude $\lambda$. Its generating vector field is $F=X+\lambda V$, where $X$ is the geodesic vector field on $SM$ and $V$ is the vertical vector field. The system $(M, g, \lambda)$ is called a \textit{thermostat}. 

If $\lambda$ does not depend on velocity, that is, if it corresponds to a function on $M$, then $\varphi_t$ is the \textit{magnetic flow} associated with the magnetic field $\lambda \mu_a$, where  $\mu_a$ is the area form of $(M,g)$. When $\lambda$ depends linearly on velocity, that is, when it corresponds to a $1$-form on $M$, we instead obtain a \textit{Gaussian thermostat}, which is reversible in the sense that the flip $(x,v)\mapsto (x, -v)$ on $SM$ conjugates $\varphi_t$ with $\varphi_{-t}$ (just as in the case of geodesic flows). The resulting flows are interesting from a dynamical point of view because, in contrast to geodesic or magnetic flows, they may not preserve any absolutely continuous measure (see \cite{dairbekov07}). Gaussian thermostats also appear in geometry as the geodesic flows of metric connections, including those with non-zero torsion (see \cite{przytycki08}). We thus think of them as dissipative geodesic flows. 

We are interested in rigidity results for thermostats satisfying the Anosov property. By \cite[Theorem A]{ghys84}, these flows are topologically orbit equivalent to the geodesic flow of any metric with constant negative curvature on $M$ via a Hölder homeomorphism which is in fact isotopic to the identity. In particular, this tells us that the flows of thermostats are transitive and topologically mixing, so the idea is that the richness of the chaotic orbits should allow one to recover information about the system. 

The set-up is as follows. Given two thermostats $(M, g, \lambda)$ and $(M, \tilde{g}, \tilde{\lambda})$ on the same surface $M$, we assume there is a smooth orbit equivalence $\phi : \widetilde{S}M\to SM$ which is isotopic to the identity. Here, $\widetilde{S}M$ is the unit tangent bundle with respect to the metric $\tilde{g}$ on $M$, and orbit equivalence means that oriented orbits are mapped to oriented orbits, that is, there exists $c\in \mathcal{C}^\infty(SM, \R_{>0})$ such that $\phi_* \widetilde{F}=cF$. In particular, $\phi$ is a conjugacy if $c$ is identically $1$. There is a natural identification of $SM$ with $\widetilde{S}M$ by scaling the fibres via the map 
\begin{equation}\label{eq:scaling-map}
s: SM\to \widetilde{S}M, \qquad s(x,v)\coloneqq (x, v/\Vert v\Vert_{\tilde{g}}).
\end{equation}
By saying that $\phi$ is isotopic to the identity, we mean that $s\circ \phi :\widetilde{S}M\to \widetilde{S}M$ is isotopic to the identity in the usual sense. 

\begin{question}
If both thermostat flows are Anosov, what is the relationship, if any, between $(g, \lambda)$ and $(\tilde{g}, \tilde{\lambda})$? 
\end{question}

The work in \cite{guillarmou25} gives an answer when $\lambda=\tilde{\lambda}=0$ and $\phi$ is a conjugacy: $g$ and $\tilde{g}$ must be isometric via an isometry isotopic to the identity. Rather than assuming a smooth conjugacy isotopic to the identity, they work with the equivalent condition that the metrics have the same marked length spectrum. This equivalence also holds for magnetic flows (using \cite[Theorem 1.1]{gogolev23}). However, while the notion of marked length spectrum still makes sense for any thermostat (each non-trivial free homotopy class on $M$ contains a unique closed thermostat geodesic by \cite[Theorem A]{ghys84}), equality of marked length spectra only guarantees a Hölder continuous conjugacy.

Still with $\phi$ as a conjugacy, the paper \cite{grognet99} deals with the mixed case where $(M, g, \lambda)$ is a magnetic system and $(M, \tilde{g}, 0)$ is geodesic, but at the cost of additional assumptions: $\tilde{g}$ has negative Gaussian curvature, $M$ has the same area with respect to $g$ and $\tilde{g}$, and neither $\lambda$ nor its first derivative are too big. The conclusion is then that $g$ and $\tilde{g}$ are isometric via an isometry isotopic to the identity and that $\lambda=0$. 

More recently, progress has been made in \cite{marshall-reber23, marshall-reber24} to understand a deformative version of our question in the purely magnetic case, framed through the lens of marked length spectrum rigidity. 

\subsection{Main results} Beyond its physical motivation, the magnetic case represents the first step towards the broader goal of understanding thermostats: it corresponds to the case where $\lambda=\lambda_0$ has Fourier degree $0$ (see \S\ref{subsection:Fourier}). Our first result is the following theorem.

\begin{theorem}\label{theorem:second-theorem}
Let $(M, g, \lambda)$ and $(M, \tilde{g}, \tilde{\lambda})$ be two Anosov magnetic systems. If there is a smooth orbit equivalence isotopic to the identity between them, then there exists a smooth diffeomorphism $\psi : M\to M$ isotopic to the identity such that $\psi^*\tilde{g}=e^{2f}g$ for some $f\in \mathcal{C}^\infty(M, \R)$. Moreover, if the orbit equivalence is a conjugacy and $f=0$, then  $\lambda=0$ if and only if $\tilde{\lambda}=0$. 
\end{theorem}

We note that finding a relationship between $\lambda$ and $\tilde{\lambda}$ in the general magnetic case remains an open question. A key similarity between geodesic and magnetic flows is that they preserve the Liouville measure on $SM$. As we will explain, this allows most of the key arguments from the paper  \cite{guillarmou25} to also go through in the magnetic case.

For this reason, the main emphasis of this paper is instead on Gaussian thermostats. These correspond to the case where $\lambda=\lambda_{-1}+\lambda_1$ or, equivalently, $\lambda=\ell_1\theta$ for some 1-form $\theta$ on $M$, where
\begin{equation}\label{eq:defn-ell_1}
(\ell_1\theta)(x,v)\coloneqq\theta_x(v)
\end{equation}
denotes the restriction to $SM$ of smooth differential forms (so that we may see them as functions on $SM$). We will denote a Gaussian thermostat $(M, g, \lambda)$ by $(M, g, \theta)$ to highlight its particular form. 

One can also study Gaussian thermostats using  an \textit{external vector field} $E$. This is the vector field on $M$ uniquely characterized by $\theta_x(v)=g_x(E(x),Jv)$, that is, the vector field dual to $\star \theta$, where $\star$ is the Hodge star operator of the metric $g$.

As we allow $\lambda$ to have Fourier degree 1, we introduce the possibility of new dynamical features absent from the geodesic and magnetic cases. For instance, by \cite[Theorem A]{dairbekov07}, a Gaussian thermostat preserves an absolutely continuous invariant measure on $SM$ if and only if $\star\theta$ is exact. This means in particular that the Liouville measure may no longer be preserved and it allows for fractal Sinai--Ruelle--Bowen measures.

The \textit{thermostat curvature} of $(M, g, \theta)$ is the quantity
\begin{equation}\label{eq:thermostat-curvature}
\mathbb{K}= \pi^*(K_g -\div_{\mu_a} E),
\end{equation}
where $K_g$ is the Gaussian curvature of $(M, g)$. If $\mathbb{K}<0$, then the flow is Anosov by \cite[Theorem 5.2]{wojtkowski00b}, in analogy with the geodesic case. Note that equation \eqref{eq:thermostat-curvature} is a particular case of the more general definition
\begin{equation}\label{eq:general-thermostat-curvature}
\mathbb{K}\coloneqq\pi^*K_g -H\lambda+\lambda^2+FV\lambda
\end{equation}
used for any thermostat $(M, g, \lambda)$.

This leads us to our next main result.
\begin{theorem}\label{theorem:main-theorem}  
Let $(M, g, \theta)$ and $(M, \tilde{g}, \tilde{\theta})$ be two Gaussian thermostats with $\mathbb{K},\widetilde{\mathbb{K}}<0$. If there is a smooth orbit equivalence isotopic to the identity between them, then there is a smooth diffeomorphism $\psi : M\to M$ isotopic to the identity such that $\psi^*\tilde{g}=e^{2f}g$ for some $f\in \mathcal{C}^\infty(M, \R)$. Moreover, if $\star\theta$ or $\tilde{\star}\tilde{\theta}$ is closed, then $\star (\psi^*\tilde{\theta}-\theta)$ is exact.
\end{theorem}

As shown in Lemma \ref{lemma:rescaling}, the scaling map defined in equation \eqref{eq:scaling-map} yields a smooth orbit equivalence isotopic to the identity between the Gaussian thermostats
$(M, g, \theta)$ and $(M, e^{2f}g, \theta+\star df)$, with a time-change by $e^f$. This implies that the conformal factor in our main result is optimal and that it is necessary to leave room for an exact difference when relating the $1$-forms. However, it is unclear at this stage whether the closedness condition is really necessary to establish this last relationship.

Ideally, one would like to extend this result to the general Anosov case. The only place where we use the negative thermostat curvature is in showing that the Gaussian thermostats satisfy the attenuated tensor tomography problem of order $1$ (see \S\ref{subsection:tensor-tomography}). We do not have this issue in the purely magnetic case, which is why we were able to simply assume the more general Anosov property in Theorem \ref{theorem:second-theorem}. Removing the negative thermostat curvature assumption should also allow one to mix the magnetic case with Gaussian thermostats, that is, to take $\lambda=\lambda_{-1}+\lambda_0+\lambda_1$. 

As pointed out above, there are still open questions regarding the rigidity of $\lambda$ for $\lambda$ of Fourier degree $1$. It is also unclear at this stage how much information is gained from having a genuine conjugacy versus an orbit equivalence, and whether the conjugating  diffeomorphism  $\phi$ itself must have some particular form as in the purely geodesic case (see \cite[Corollary 1.3]{guillarmou25}).

After this work, a natural question is whether anything can be said for $\lambda$ of Fourier degree $\geq 2$. As we show with the no-go Lemma \ref{lemma:no-go-lemma}, the current argument does not work for these thermostats. However, there are interesting examples of such systems. For instance, when $\lambda$ is the real part of a holomorphic differential of degree $\geq 2$, the corresponding thermostat admits an interpretation as coupled vortex equations (see \cite{mettler19, mettler22}). It was also shown in \cite{mettler20} that the geodesic flow of an affine connection on $M$ is, up to a time-change, the flow of a thermostat with $\lambda$ of the form $\lambda=\lambda_{-3}+\lambda_{-1}+\lambda_1+\lambda_3$. Just as Theorem \ref{theorem:second-theorem} implies that an Anosov magnetic system $(M, g, \lambda)$ with $\lambda\neq 0$ cannot be smoothly conjugate to an Anosov geodesic flow $(M, g, 0)$ with the same metric by a conjugacy isotopic to the identity, it would be interesting to further categorize thermostats. 

Finally, we note that Theorem \ref{theorem:extension-forms}, which applies to thermostats, was placed in the appendix to improve the overall exposition of the paper, but it represents a new result related to the injectivity of the thermostat X-ray transform. 

\subsection{Strategy} Our main inspiration is the approach in \cite{guillarmou25}. Indeed, we show that a smooth orbit equivalence isotopic to the identity determines the complex structure of the metric $g$ up to biholomorphisms isotopic to the identity (Proposition \ref{proposition:same-teichmuller-class}). This allows us to conclude that the two metrics $g$ and $\tilde{g}$ must be conformally equivalent via a smooth diffeomorphism of $M$ isotopic to the identity. 

To show that the orbit equivalence determines the complex structure, we rely on Torelli's theorem (Theorem \ref{theorem:torelli}), which tells us that it is enough to show that the period matrix of the underlying Riemann surface is preserved. To be able to conclude that the resulting diffeomorphism is isotopic to the identity, we use the fact that the argument can be repeated on any finite cover. 

The period matrix is defined in terms of holomorphic $1$-forms on $M$. We show with Theorem \ref{theorem:extension-holomorphic} that these can always be associated to the first Fourier modes of certain distributions $\mathcal{D}_{\text{tr},+}(SM)$ on $SM$ satisfying a transport equation and with non-negative Fourier modes (see \S\ref{subsubsection:fibrewise-holomorphic-distributions}). Asking for these distributions to only have  non-negative Fourier modes is a critical requirement for the rest of the argument, but it does not carry over to the case of thermostats when $\lambda$ has Fourier degree $\geq 2$.

We then establish in Lemma \ref{lemma:pairing-two} a pairing formula showing that the integral of any holomorphic $1$-form over a thermostat geodesic $\gamma$ on $M$ (that is, the periods of the period matrix) is the same as the integral over $\pi^{-1}(\gamma)\subset SM$ of an associated $2$-current  invariant by $F$ and living in a certain subspace $\mathcal{F}(SM)$ (see \S\ref{subsubsection:divergence}). This pairing formula then tells us that the smooth orbit equivalence preserves the period matrix.

At a high level, there are two main challenges and departures from \cite{guillarmou25}: the first is in handling a general orbit equivalence instead of a conjugacy, and the second is in dealing with the fact that Gaussian thermostats may not be volume-preserving.

The presence of a non-zero divergence with respect to the Liouville volume form manifests in several ways. First, instead of flow-invariant distributions, the right object of study becomes solutions to the dual transport equation. This subspace is no longer preserved by the pullback of the orbit equivalence, so we have to introduce the space $\mathcal{F}(SM)$ of  $2$-currents mentioned above and establish a one-to-one relationship with the distributions solving the transport equation (Lemma \ref{lemma:current-isomorphism}). We then have to check that the wavefront set analysis is unaffected by factoring the correspondence through this space (Lemma \ref{lemma:wavefrontset-pullback}) and that $\mathcal{D}_{\text{tr},+}(SM)$ is mapped to  $\mathcal{D}_{\text{tr},+}(\widetilde{S}M)$ (Proposition \ref{proposition:action-1}). 
 
 Another complication due to the dissipation is in showing that any holomorphic 1-form can be seen as the first Fourier mode of an element in $\mathcal{D}_{\text{tr}, +}(SM)$, as previously mentioned. The heavy lifting to address this issue is done in Appendix \ref{appendix:A}. Furthermore, again due to the divergence, we have to explain why Gaussian thermostats with negative thermostat curvature satisfy the attenuated tensor tomography problem of order 1 (Theorem \ref{theorem:thermostat-tensor-tomography}).
 
For the pairing formula previously described, we have replaced the role of the Liouville form with that of a certain form defined in equation  \eqref{eq:definition-special-form}. Finally, to relate $\lambda$ with $\tilde{\lambda}$, we rely on new arguments which, at their core, involve the smooth Livšic theorem.

\subsection{Organization of the paper} 
In  \S\ref{section:preliminaries}, we introduce the background tools necessary for the rest of the paper. Specifically, \S\ref{subsection:geometry-of-sm} provides a short introduction to the geometry and vertical Fourier analysis of the unit tangent bundle. It also introduces the new objects needed to deal with the divergence of the thermostats. In \S\ref{subsection:complex-geometry}, we review the complex geometry and harmonic analysis on a surface, while \S\ref{subsection:hyperbolic-dynamics} delves into hyperbolic dynamics and tensor tomography. 

In \S\ref{section:action-holomorphic-differentials}, we explain how a smooth orbit equivalence acts on holomorphic differentials and we establish the pairing formula needed to show that period matrices are preserved. We then present the proofs of our main results in \S\ref{section:end-of-the-proofs}. 

Appendix \ref{appendix:A} delves into the question of finding distributional solutions, with prescribed Fourier modes, of the relevant transport equation for a thermostat.

\subsection*{Acknowledgements} I would like to thank my supervisor, Gabriel Paternain, for suggesting this project and guiding me while working on it. This research was supported by a Harding Distinguished Postgraduate Scholarship.

\section{Preliminaries}\label{section:preliminaries}
In what follows, $(M,g)$ is a closed oriented Riemannian surface and we take an arbitrary $\lambda\in \mathcal{C}^\infty(SM, \R)$. Whenever we use additional assumptions, it will be clearly stated in the result statements. We will sometimes need a second thermostat $(M, \tilde{g}, \tilde{\lambda})$. All the objects depending on the metric will then be labelled accordingly. Finally, we denote by $\phi : \widetilde{S}M\to SM$ a smooth orbit equivalence between the thermostats $(M, \tilde{g}, \tilde{\lambda})$ and $(M, g, \lambda)$. Once again, we will specify when we assume it to be isotopic to the identity.

\subsection{Unit tangent bundle of the surface}\label{subsection:geometry-of-sm}
We review some basics concerning the unit tangent bundle of $M$, that is, the $3$-dimensional manifold defined as
$$SM\coloneqq\left\{(x,v)\in TM\mid \Vert v\Vert_g = 1\right\},$$
together with its natural projection $\pi: SM\to M$.
\subsubsection{Geometry of $SM$}
As previously, let $X$ be the geodesic vector field on $SM$ and let $V$ be the vertical vector field generating the circle action on the fibres. We define $H\coloneqq[V,X]$. Then, $(X, H, V)$ is a positively oriented global frame for $T(SM)$, orthonormal with respect to the \textit{Sasaki metric} (the natural lift of $g$ to $SM$). We set $\mathbb{H}\coloneqq\R H$ and $\mathbb{V}\coloneqq\R V$. We also note that the geodesic vector field splits into $X=\eta_+ +\eta_-$, where $\eta_\pm$ are the \textit{raising} and \textit{lowering Guillemin--Kazhdan operators} given by
\begin{equation}\label{eq:ladder-operators}
\eta_\pm \coloneqq \dfrac{1}{2}(X\mp i H).
\end{equation}

The \textit{Liouville} $1$-form $\alpha\in \mathcal{C}^\infty(SM, T^*(SM))$ is defined by the relations $\alpha(X)=1$ and $\alpha(H)=\alpha(V)=0$. It is invariant by the geodesic flow in the sense that $\mathcal{L}_X \alpha=0$. The $2$-form $d\alpha$ is non-degenerate on the \textit{contact plane} $\mathbb{H}\oplus \mathbb{V}$ and it satisfies $\iota_X d\alpha=0$. Hence,
$$\mu: = -\alpha\wedge d\alpha$$
is a nowhere-vanishing volume form invariant by the geodesic flow. We call it the \textit{Liouville} volume form (and also use $\mu$ to denote the induced Liouville measure on $SM$). It coincides with the Riemannian volume form induced by the Sasaki metric on $SM$. From now on, the $L^2$ space on $SM$ is defined as $L^2(SM)\coloneqq L^2(SM, \mu)$. 

We also define the $1$-forms $\beta, \psi $ on $SM$ by the relations $\beta(H)=1=\psi(V)$ and $\beta(X)=\beta(V)=0=\psi(X)=\psi(H)$.  It is easy to check that $d\alpha = \psi \wedge\beta$ so that $\mu = \alpha\wedge \beta \wedge\psi$. We set $(\R X)^\ast\coloneqq\R\alpha$, $\mathbb{H}^\ast \coloneqq \R\beta$ and $\mathbb{V}^\ast \coloneqq\R\psi$. 
We refer to \cite[\S 3.5]{paternain23b}  for further details on the geometric structure of $SM$. 

\subsubsection{Appearance of the divergence}
The key difference between thermostats and geodesic or magnetic flows is that the generating vector field $F$ might not preserve the Liouville volume form $\mu$. Recall that the \textit{divergence} of the vector field $F$ with respect to the volume form $\mu$ is the function $\div_\mu F\in \mathcal{C}^\infty(SM, \R)$ uniquely characterized by 
$$\mathcal{L}_F \mu = (\div_\mu F)\mu.$$
 The following result is proved in \cite[Lemma 3.2]{dairbekov07}. 
\begin{lemma}\label{lemma:divergence}
Let $(M, g, \lambda)$ be a thermostat. Then, we have
$$\mathcal{L}_F\mu = V(\lambda)\mu, \qquad \mathcal{L}_H\mu = 0, \qquad \mathcal{L}_V \mu=0.$$
\end{lemma}

In the geodesic and magnetic cases, we have $V\lambda=0$, so the Liouville volume form is preserved. Another way in which the divergence manifests itself is when calculating the adjoint operators with respect to the $L^2$ inner product on $SM$:
$$F^*=-(F+V\lambda),\qquad H^*=-H, \qquad V^*=-V.$$
This is relevant when extending differential operators to act on the space of distributions. Recall that any differential operator $P$ with smooth real-valued coefficients acts on a distribution $u\in \mathcal{D}'(SM)$ by duality, that is, $\langle P u , \varphi\rangle_{\mathcal{D}'(SM)}\coloneqq\langle u, P^* \varphi \rangle_{\mathcal{D}'(SM)}$ for any $\varphi\in \mathcal{C}^\infty(SM)$. The subspace of distributional solutions to the transport equation
\begin{equation*}
\mathcal{D}'_{\text{tr}}(SM)  \coloneqq\left\{u\in \mathcal{D}'(SM) \mid (F+V\lambda)u=0\right\}
\end{equation*}
thus corresponds to the distributions $u\in \mathcal{D}'(SM)$ such that $\langle u , F\varphi\rangle_{\mathcal{D}'(SM)}=0$ for all $\varphi\in \mathcal{C}^\infty(SM)$. If $V\lambda=0$, these are simply the distributions invariant by the flow. 

\subsubsection{Divergence and smooth orbit equivalences}\label{subsubsection:divergence}

It is crucial to understand how the divergence of a system interacts with smooth orbit equivalences. 

The next result, which we have stated in a broader setting than the one we are studying in this paper to highlight its generality, relates the divergences of two flows associated by a smooth orbit equivalence. 
\begin{lemma}\label{lemma:relation-divergences}
Let $N$ and $\widetilde{N}$ be two oriented manifolds endowed with nowhere-vanishing volume forms $\mu$ and $\tilde{\mu}$, and smooth nowhere-vanishing vector fields $Y$ and $\widetilde{Y}$. Suppose that $\phi : \widetilde{N}\to N$ is a smooth orbit equivalence between the flows generated by $\widetilde{Y}$ and $Y$. If we write $\phi_*\widetilde{Y}=cY$ with $c \in \mathcal{C}^\infty(N, \R_{>0})$ and $\phi^*\mu = (\det \phi )\tilde{\mu}$ with $\det \phi\in \mathcal{C}^\infty(\widetilde{N}, \R)$, then $$(\det \phi)\phi^*(\div_\mu Y)=\widetilde{Y}\left(\dfrac{\det \phi}{\phi^*c}\right)+\dfrac{\det \phi}{\phi^*c}\div_{\tilde{\mu}}\widetilde{Y}.$$
In particular, if $\phi$ preserves the orientation, that is, $\det \phi >0$, then  $$\phi^*(c)\phi^*(\div_\mu Y)=\widetilde{Y}\left(\ln \left(\dfrac{\det \phi}{\phi^*c}\right)\right)+\div_{\tilde{\mu}} \widetilde{Y}.$$
\end{lemma}
\begin{proof}
We compute 
\begin{equation*}
\begin{alignedat}{1}
\phi^*(\mathcal{L}_Y \mu)&=\phi^*(d(\iota_Y \mu))\\
&=d(\iota_{\phi^{-1}_* Y} \phi^*(\mu))\\
&=d\left(\dfrac{\det \phi}{\phi^*c}\iota_{\widetilde{Y}} \tilde{\mu}\right)\\
&=d\left(\dfrac{\det \phi}{\phi^*c}\right)\wedge \iota_{\widetilde{Y}} \tilde{\mu}+\dfrac{\det \phi}{\phi^*c}d(\iota_{\widetilde{Y}}\tilde{\mu}) \\
&=\left(\widetilde{Y}\left(\dfrac{\det \phi}{\phi^*c}\right)+\dfrac{\det \phi}{\phi^*c}\div_{\tilde{\mu}}\widetilde{Y}\right)\tilde{\mu}.
\end{alignedat}
\end{equation*}
However, we also have
\begin{equation*}
\begin{alignedat}{1}
\phi^*(\mathcal{L}_Y \mu)&=\phi^*(\div_\mu Y \mu)\\
&=\phi^*(\div_\mu Y)\phi^*\mu\\
&=\phi^* (\div_\mu Y)(\det \phi)\tilde{\mu},
\end{alignedat}
\end{equation*}
so putting these together yields the desired result since $\tilde{\mu}$ is nowhere-vanishing.
\end{proof}

In the geodesic and magnetic cases, the pullback $\phi^*$ of the smooth orbit equivalence $\phi: \widetilde{S}M\to SM$ sends the space $\mathcal{D}'_{\text{tr}}(SM)$ to $\mathcal{D}'_{\text{tr}}(\widetilde{S}M)$. More generally, however, the divergence term $V\lambda$ appearing in the transport equation breaks this down. 

Instead, a more useful perspective is to look at the following subspace of $2$-currents (or distributional $2$-forms) on $SM$ invariant by $F$: 
\begin{equation*}\label{eq:2-currents}
\mathcal{F}(SM)\coloneqq \left\{\sigma \in \mathcal{D}'(SM, \Lambda^2 T^*(SM))\mid \iota_F \sigma = d\sigma = 0\right\}.
\end{equation*}
This set only depends on the foliation corresponding to $F$, that is, it is invariant under time-changes, so we get a $\mathbb{C}$-linear isomorphism $\phi^*: \mathcal{F}(SM)\to \mathcal{F}(\widetilde{S}M)$. The 2-form
\begin{equation}\label{eq:definition-special-form}
\omega\coloneqq\iota_F \mu
\end{equation}
then allows us to establish a relationship with solutions to the transport equation. 
\begin{lemma}\label{lemma:current-isomorphism}
The map $L:\mathcal{D}'_{\textup{tr}}(SM)\to \mathcal{F}(SM)$ given by $u\mapsto u \omega$ is a $\mathbb{C}$-linear isomorphism.
\end{lemma}
\begin{proof}
Using Cartan's magic formula and Lemma \ref{lemma:divergence}, note that
\begin{equation*}
\begin{alignedat}{1}
d(u\omega)&= du\wedge \omega + u\, d\omega\\
&=du\wedge(\iota_F\mu)+ u \mathcal{L}_F\mu\\
&=\iota_F(du)\mu +uV(\lambda)\mu\\
&=(Fu+V(\lambda)u)\mu.
\end{alignedat}
\end{equation*}
Therefore, the $2$-current $u\omega$ is closed if and only if $(F+V\lambda)u=0$. Since $F$ is nowhere-vanishing, any $2$-current $\sigma$ on $SM$ satisfying $\iota_F \sigma = 0$ must be of the form $\sigma=u\omega$ for some distribution $u\in \mathcal{D}'(SM)$.
\end{proof}
Thanks to this identification, we can now define a map $\Phi: \mathcal{D}'_{\text{tr}}(SM)\to\mathcal{D}'_{\text{tr}}(\widetilde{S}M)$ associated to the smooth orbit equivalence $\phi : \widetilde{S}M\to SM$ via the following diagram:
\begin{equation}\label{eq:definition-phi}
\begin{tikzcd}
\mathcal{F}(SM)\arrow[r, "\phi^*"] & \mathcal{F}(\widetilde{S}M)\arrow[d, "\widetilde{L}^{-1}"]\\
\mathcal{D}_{\text{tr}}'(SM)\arrow[r, "\Phi", dashrightarrow]\arrow[u, "L"] & \mathcal{D}_{\text{tr}}'(\widetilde{S}M)
\end{tikzcd}
\end{equation}
This point of view does not affect the wavefront set analysis. As a reminder, the \textit{wavefront set} $\text{WF}(u)$ of a distribution $u\in \mathcal{D}'(SM)$ describes the directions in $T^*(SM)$ along which the distribution is not smooth. It refines the notion of the singular support of $u$, which only records the points in $SM$ where $u$ is singular. For a detailed introduction, we refer the reader to \cite[Chapter 8]{hormander03}.
\begin{lemma}\label{lemma:wavefrontset-pullback}
If $\phi$ preserves the orientation, then, for all $u\in \mathcal{D}'_{\textup{tr}}(SM)$, we have
$$\textup{WF}(\Phi u)=\textup{WF}(\phi^* u).$$
\end{lemma}
\begin{proof}
Let $q\in \mathcal{C}^\infty(\widetilde{S}M, \R_{>0})$ be the function such that $\phi^*\omega=q\tilde{\omega}$. Then, we get
\begin{equation*}
\begin{alignedat}{1}
\Phi u &= \widetilde{L}^{-1}\phi^* Lu \\
&=\widetilde{L}^{-1}\phi^* (u\omega) \\
&=\widetilde{L}^{-1}(\phi^*(u) q\tilde{\omega})\\
&=q\phi^* u.
\end{alignedat}
\end{equation*}
Since multiplication by the nowhere-vanishing function $q$ is elliptic, we get the result by elliptic regularity (see \cite[Theorem 8.3.2]{hormander03}). 
\end{proof}

By the properties of wavefront sets under pullback operators (see \cite[Theorem 8.2.4]{hormander03} for instance), we thus obtain
$$\textup{WF}(\Phi u)= d\phi^\top(\textup{WF}(u))$$
for all $u\in \mathcal{D}'_{\textup{tr}}(SM)$, where $d\phi^\top: T^*(SM)\to T^*(\widetilde{S}M)$ is the symplectic lift of $\phi^{-1}$ to the cotangent bundles given by
$$d\phi^\top(v, \xi) \coloneqq\left(\phi^{-1}(v), d\phi^\top_{\phi^{-1}(v)} \xi\right), \qquad (v, \xi)\in T^*(SM).$$
\subsubsection{Vertical Fourier expansions}\label{subsection:Fourier} The space $\mathcal{C}^\infty(SM)$ breaks up as 
$$\mathcal{C}^\infty(SM)=\oplus_{k\in \Z}\Omega_k, \qquad \Omega_k \coloneqq\left\{ u\in \mathcal{C}^\infty(SM)\mid \, V u = iku\right\}.$$
This decomposition is orthogonal with respect to the $L^2$ inner product on $SM$ and with $\mathcal{C}^\infty$ being replaced by $L^2$. For any $u\in \mathcal{C}^\infty(SM)$, we shall write $u=\sum_{k\in \Z}u_k$, where the \textit{$k$th Fourier mode} $u_k\in \Omega_k$ is given by
\begin{equation}\label{eq:definition-fourier-mode}
u_k(x,v)\coloneqq\dfrac{1}{2\pi}\int_0^{2\pi}u(\rho_t(x,v))e^{-ikt}\, dt,
\end{equation}
with $\rho_t$ being the flow generated by $V$. More generally, any distribution $u\in \mathcal{D}'(SM)$ can be decomposed as $u=\sum_{k\in \Z}u_k$, where each $u_k\in \mathcal{D}'(SM)$ is defined by 
$$\langle u_k , \varphi \rangle \coloneqq \langle u, \varphi_{-k}\rangle_{\mathcal{D}'(SM)}, \qquad \varphi\in \mathcal{C}^\infty(SM)$$
and satisfies $Vu_k = iku_k$. 

If a distribution on $SM$ only has finitely many non-trivial Fourier modes, we say that it has \textit{finite Fourier degree}. The smallest $m\in \N$ such that $u_k=0$ for all $|k|> m$ is then called the \textit{Fourier degree} of $u$. 

It also worth noting that the ladder operators $\eta_\pm$ in equation \eqref{eq:ladder-operators} take their name from the fact that they act as raising/lowering operators on the Fourier decomposition, that is, 
$$\eta_\pm : \Omega_k\to \Omega_{k\pm 1}$$
for all $k\in \Z$. In particular, we have $(Xu)_k = \eta_-u_{k+1}+\eta_+u_{k-1}$ for any $u\in \mathcal{D}'(SM)$. 

\subsection{Complex geometry}\label{subsection:complex-geometry} The conformal class of the Riemannian metric $g$ and the orientation of $M$ induce a complex structure $J: TM\to TM$ on $M$, making it into a Riemann surface which we denote by $(M, J)$. 
\subsubsection{Complex structures.} The \textit{Teichmüller space} of $M$, denoted by $\mathcal{T}(M)$, is the space of complex structures on $M$ modulo the equivalence relation that $J\sim \widetilde{J}$ if and only if there exists a diffeomorphism $\psi : M\to M$ isotopic to the identity such that $\psi^*\widetilde{J}=J$. We will denote such an equivalence class of complex structures by $[J]$. 

The \textit{mapping class group} $\text{MCG}(M)$ is defined as the quotient of orientation-preserving diffeomorphisms on $M$ modulo isotopy. They act on $\mathcal{T}(M)$ by pullback, and the quotient space $\mathcal{M}(M)\coloneqq\mathcal{T}(M)/\text{MCG}(M)$ is the \textit{moduli space of complex structures} on $M$. See \cite{farb12} for a thorough introduction. 

Each complex structure $J$ determines a \textit{canonical line bundle} $\kappa : = T^\ast_{1,0} M$ and its conjugate $\overline{\kappa} : = T^\ast_{0,1} M$  over $M$. Locally, the sections of $\kappa$ have the form $w(z)\, dz$ while the sections of $\overline{\kappa}$ have the form $w(z)\, d\overline{z}$. The dual of $\kappa$ is called the \textit{anti-canonical line bundle} and we identify it with $\overline{\kappa}$ by using the Hermitian inner product on $\kappa$ induced by the Riemannian metric on $M$. We then denote this bundle by $\kappa^{-1}$ so that the tensor powers $\kappa^{\otimes k}$ make sense for all $k\in \Z$. 

We will denote by $H^0_J(M, \kappa^{\otimes k})$ the space of $J$-holomorphic sections of the $k$th tensor power of the canonical line bundle $\kappa$. Locally, its elements have the form $w(z)\, dz^k$ for $k\geq 0$ and $w(z)\, d\bar{z}^{-k}$ for $k<0$. 

\subsubsection{Fiberwise holomorphic distributions}\label{subsubsection:fibrewise-holomorphic-distributions} Each subspace $\Omega_k$ of Fourier modes can be identified with $\mathcal{C}^\infty(M, \kappa^{\otimes k})$, the set of smooth sections of the bundle $\kappa^{\otimes k}$ \cite[Lemma 6.1.19]{paternain23b}. Indeed, we have a $\mathbb{C}$-linear isomorphism
$$\pi_k^\ast :\mathcal{C}^\infty(M, \kappa^{\otimes k})\to \Omega_k$$
given by restriction to $SM$, that is, locally, for $k\geq 0$, 
$$\pi^\ast_k(w\, dz^k)(x,v)=w(x)\, (dz(v))^k.$$
Note that the definition of $\pi_k^\ast$ depends on the choice of the metric $g$. We denote by $\pi_{k\ast}$ its $L^2$ adjoint. Locally, for $k\geq 0$, we have
$$(\pi_{k\ast} u)(x)=\left(\int_{S_x M}u(x, \cdot)\psi\right)\, dz^k.$$
Once we extend the operators to distributions by duality, the projection onto the $k$th Fourier mode is simply given by $(2\pi)^{-1}\pi^\ast_k \pi_{k\ast}$ acting on $\mathcal{D}'(SM)$. 

Under this identification, we can essentially think of the raising/lowering operators $\eta_\pm$ as $\partial$ and $\bar{\partial}$ operators thanks to the following result (see \cite[pp.~153--154]{paternain23b}). 
\begin{lemma}\label{lemma:raising-lowering-equivalence}
For $k\geq 0$, the following diagram commutes.
\begin{equation*}
\begin{tikzcd}
\mathcal{C}^\infty(M, \kappa^{\otimes k})\arrow[d, "\bar{\partial}" ']\arrow[r, "\pi_k^\ast"] & \Omega_k\arrow[d, "\eta_-"]\\
\mathcal{C}^\infty(M, \kappa^{\otimes k}\otimes \bar{\kappa})\arrow[r, "\pi_{k-1}^\ast"] & \Omega_{k-1}
\end{tikzcd}
\end{equation*}
For $k\leq 0$, the operator $\pi_k^*$ also intertwines the operators $\partial$ and $\eta_+$. 
\end{lemma}
As a result, for $k\geq 0$, the operator $\pi_k^*$ gives us an identification
$$H^0_J (M, \kappa^{\otimes k})\cong \Omega_k \cap \ker \eta_-.$$

We also introduce the following terminology.
\begin{definition}
A distribution $u\in \mathcal{D}'(SM)$ is said to be \textit{fibrewise holomorphic}  if $u_k=0$ for all $k<0$.
\end{definition}
Equivalently, if we define the \textit{Szegő projectors} $S_\pm : \mathcal{D}'(SM)\to \mathcal{D}'(SM)$ by $$S_+u \coloneqq \sum_{k\geq 0}u_k, \qquad S_-u \coloneqq \sum_{k\leq 0}u_k,$$
then a distribution $u$ is fibrewise holomorphic if and only if $S_+ u = u$. The projectors satisfy the commutation relations
\begin{equation}\label{eq:commutation}
[S_+, X]u=\eta_+ u_{-1}-\eta_- u_0, \qquad [S_-, X]u=\eta_- u_1-\eta_+ u_0.
\end{equation}
We will be interested in the family of fibrewise holomorphic distributions that satisfy the transport equation:
\begin{equation}\label{eq:invariant-fibrewiseholomorphic}
\mathcal{D}'_{\text{tr},+}(SM)  \coloneqq\left\{u\in \mathcal{D}'(SM) \mid (F+V\lambda)u=0, \, S_+ u = u\right\}.
\end{equation}
\subsubsection{Torelli's theorem} The complex vector space $H^0_J(M, \kappa)$ of $J$-holomorphic $1$-forms has the same dimension as the genus of $M$ (see \cite[Proposition III.2.7]{farkas92}). Given a canonical basis $\{a_j, b_j\}$ of the homology $H_1(M;\Z)$ on $M$, the following result gives us the existence of a useful basis (see \cite[Proposition, p. 63]{farkas92}). 

\begin{proposition}
There exists a unique basis $\{\zeta_j\}$ for $H^0_J(M, \kappa)$ with the property
\begin{equation}\label{eq:basis-condition}
\int_{a_j}\zeta_k = \delta_{jk}.
\end{equation}
Furthermore, the matrix $\Pi(J)$ with $(j,k)$-entry
$$(\Pi(J))_{jk}\coloneqq\int_{b_j}\zeta_k$$
is symmetric with positive definite imaginary part. 
\end{proposition}
The space of symmetric matrices with positive definite imaginary part and size given by the genus of $M$ is called the \textit{Siegel upper half-space} $\mathcal{H}(M)$. We thus get a well-defined \textit{period matrix} map 
$$\Pi : \mathcal{T}(M)\to \mathcal{H}(M).$$ 

The following form of Torelli's theorem tells us that period matrices capture a lot of information about the complex structure. 
\begin{theorem}\label{theorem:torelli}
Assume that $M$ has genus $\geq 2$. If $\Pi(J)=\Pi(\widetilde{J})$, then there exists an orientation-preserving diffeomorphism $\psi: M\to M$ such that $\psi^*\widetilde{J}=J$. 
\end{theorem}
We refer to \cite[Theorem III.12.3]{farkas92} for a proof. 

\subsection{Hyperbolic dynamics}\label{subsection:hyperbolic-dynamics} We now further assume that the flow of the thermostat $(M, g, \lambda)$ is \textit{Anosov} (or \textit{uniformly hyperbolic}). 
\subsubsection{Definition} Recall that the Anosov property means that there exist a flow-invariant continuous splitting 
$$T(SM)=\R F \oplus E_s\oplus E_u$$
and uniform constants $C\geq 1$ and $0<\rho<1$ such that for all $t\geq 0$ we have
\begin{equation}\label{eq:hyperbolic-estimates}
\left\Vert d\varphi_{t}|_{E_s}\right\Vert \leq C\rho^t, \qquad\left\Vert d\varphi_{-t}|_{E_u}\right\Vert \leq C\rho^t.
\end{equation}
In the geodesic case, the contact form $\alpha$ is preserved, so $\ker \alpha =\mathbb{H}\oplus \mathbb{V}= E_s \oplus E_u$. It is then known that $E_s\cap \mathbb{V}=\{0\}=E_u\cap \mathbb{V}$. For a thermostat, we instead know by \cite[Lemma 4.1]{dairbekov07} that
\begin{equation}\label{eq:transversality}
(\R F\oplus E_s)\cap \mathbb{V}=\{0\}=(\R F\oplus E_u)\cap \mathbb{V}.
\end{equation}
Here, $\R F\oplus E_{s/u}$ are the \textit{weak stable} and \textit{unstable subbundles}. This implies that there exist $r^{s}, r^u\in \mathcal{C}^0(SM,\R)$ such that
\begin{equation}\label{eq:pre-riccati}
Y^s\coloneqq H+r^sV\in \R F\oplus E_s, \qquad Y^u\coloneqq H+r^uV\in \R F\oplus E_u.
\end{equation}
In fact, the weak stable and unstable subbundles are $\mathcal{C}^1$ (see \cite[Corollary 1.8]{hasselblatt94}), so the functions $r^{s}$ and $r^u$ are also $\mathcal{C}^1$ (and smooth along the flow since each subbundle $\R F\oplus E_{s/u}$ is flow-invariant).  The Anosov property implies that $r^s\neq r^u$ everywhere. One may in fact show that $r^s<r^u$, so the global frame $(F, Y^s, Y^u)$ is positively oriented. See Figure \ref{figure:subbundles}. 

\begin{lemma}
Let $(M, g, \lambda)$ be an Anosov thermostat. Then, the differentiable functions $r^{s}, r^u \in \mathcal{C}^1(SM, \R)$ uniquely characterized by equation \eqref{eq:pre-riccati} satisfy $r^s<r^u$.
\end{lemma}
\begin{proof}
Since $r^s\neq r^u$ everywhere, it suffices to show the inequality at a single point. By compactness, we can pick $(x,v)\in SM$ such that $(V\lambda)(x,v)=0$. Let us define
$$\zeta(t)\coloneqq d\varphi_{-t}(V(\varphi_t(x,v))).$$
Differentiating with respect to $t$ and setting $t=0$, we obtain 
$$\dot{\zeta}(0)=[F,V](x,v).$$
Using that $[V, F]=H+V(\lambda)V$ yields
$$\dot{\zeta}(0)=-H(x,v).$$
Since $r^{s}(x,v)\neq r^{u}(x,v)$, there is a unique constant $c\in \R$ such that $V(x,v)+cX(x,v)$ belongs to $E_s\oplus E_u$. Therefore, since $E_s$ and $E_u$ are uniformly repelling and attracting sets on $E_s\oplus E_u$, respectively, we must have $r^s(x,v)<r^u(x,v)$ at this point.
\end{proof}

\begin{remark}
When $\lambda=0$ and $K_g <0$, that is, in the geodesic case with negative curvature, we have the stronger statement  $r^s < 0 < r^u$ because $[X,H]=\pi^*(K_g)V$. 
\end{remark}

The dual subbundles are defined by
$$ E_s\ast (\R F\oplus E_s)=0=E_u^\ast(\R F\oplus E_u).$$
One can check that $E_s^*$ and $E_u^*$ satisfy analogues of the estimates \eqref{eq:hyperbolic-estimates}, with $d\varphi_t$ replaced by $(d\varphi_t^{\top})^{-1}$ (inverse of the transpose). Translated to the setting of the cotangent bundle, property \eqref{eq:transversality} then becomes 
\begin{equation}\label{eq:transversality-cotangent}
E_s^\ast \cap \mathbb{H}^* = \{0\}=E_u^\ast\cap \mathbb{H}^\ast. 
\end{equation}
Further note that 
$$\Sigma=E_s^*\oplus E_u^*,$$
where 
$$\Sigma \coloneqq \left\{(v, \xi)\in T^*(SM)\mid \xi(F(v))=0\right\}$$
is the \textit{characteristic set} of the operator $F$ (usually defined without the zero section).

\begin{figure}[h]
\centering
\includegraphics[scale=0.285]{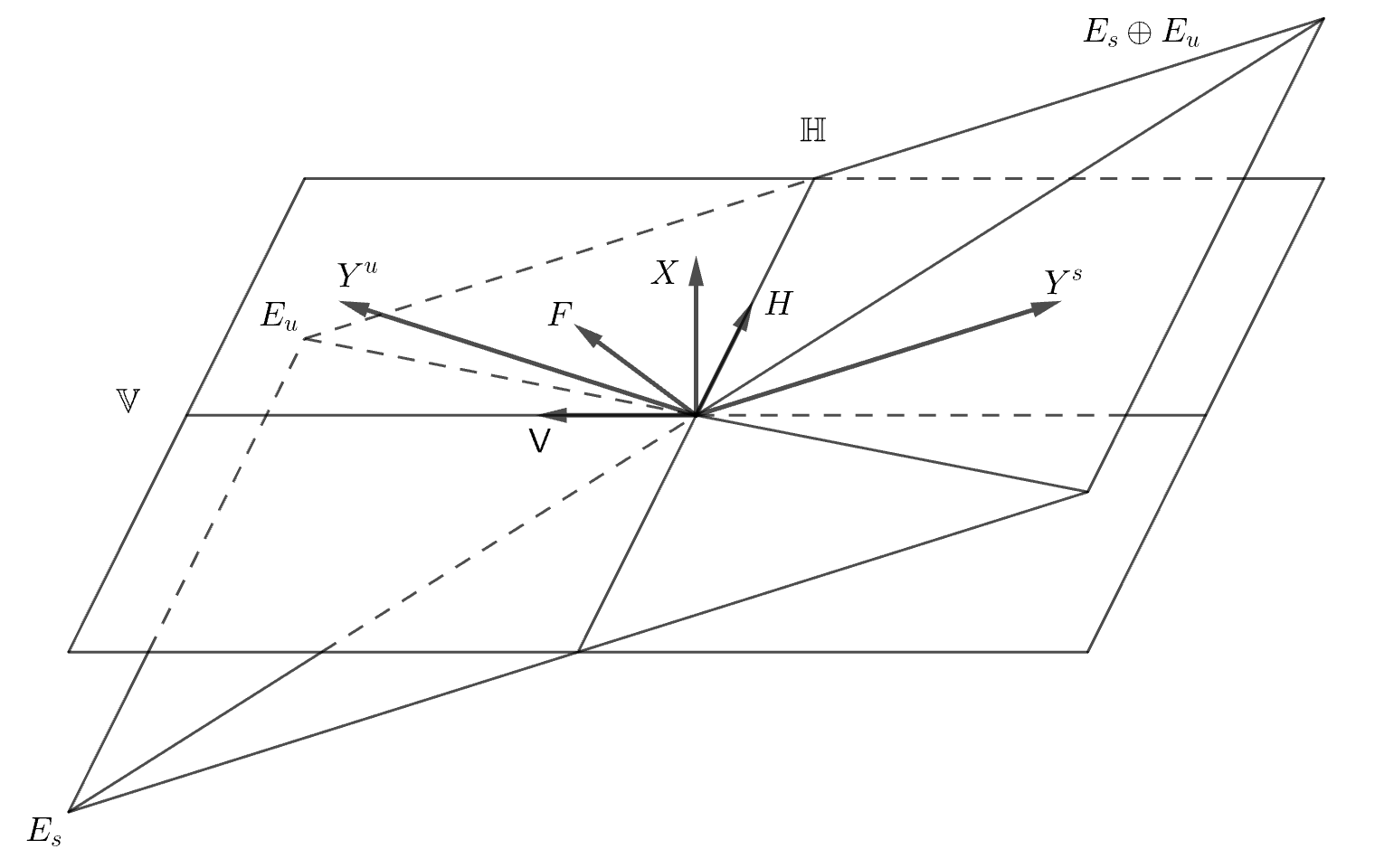}
\includegraphics[scale=0.285]{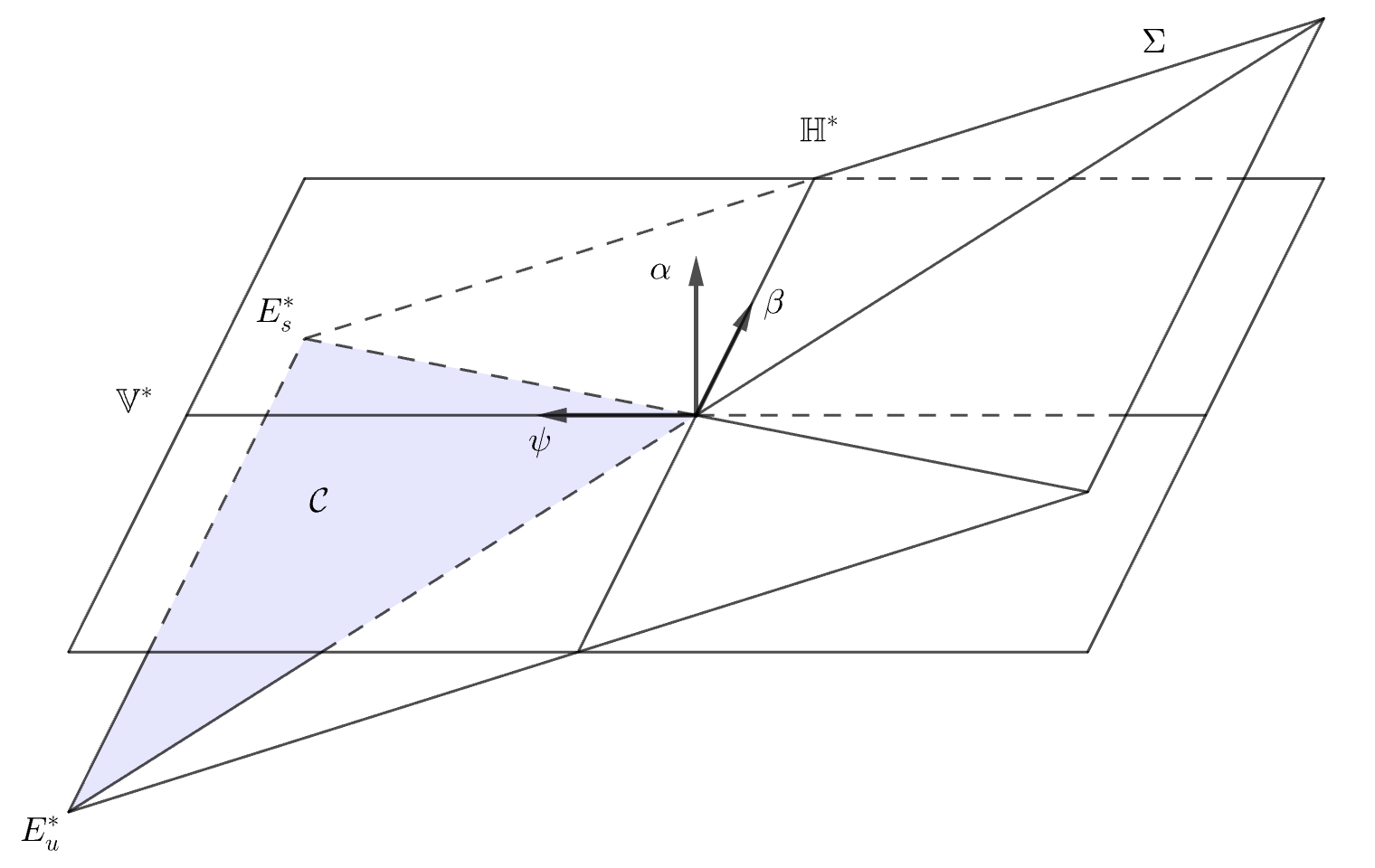}
\caption{The relevant subbundles of the tangent and cotangent bundles.}
\label{figure:subbundles}
\end{figure}
\subsubsection{Tensor tomography}\label{subsection:tensor-tomography} The tensor tomography problem is interesting in its own right, particularly as it pertains to the injectivity of the X-ray transform for thermostats. We will need the following property in the case $n=1$.
\begin{definition}\label{property}
We say that a thermostat $(M, g, \lambda)$  satisfies the \textit{attenuated tensor tomography problem of order $n$} if having $(F+V\lambda)u=f$ with $f, u\in \mathcal{C}^\infty(SM)$ and $f$ of Fourier degree $n\geq 0 $ implies that $u$ is of Fourier degree $\max(n-1,0)$.
\end{definition}

The term `attenuated' refers to the presence of the divergence $V\lambda$ in the transport equation. Such a term appears for Gaussian thermostats and thermostats of higher Fourier degree, but not for magnetic or geodesic flows.

The fact that geodesic flows satisfy the (attenuated) tensor tomography problem was first proved in negative curvature in \cite{guillemin80} for $n\geq 0$, and then generalized to the Anosov case in \cite{dairbekov03} for $n \leq 1$, \cite{paternain14} for $n\leq 2$ and \cite{guillarmou17} for $n \geq 2$. It was also shown in \cite{dairbekov05} that Anosov magnetic flows satisfy the (attenuated) tensor tomography problem of order $n\leq 1$. 

For thermostats of higher Fourier degree, the non-attenuated and attenuated versions of the tensor tomography problem are different. In \cite{dairbekov07}, it was proved that Gaussian thermostats (potentially mixed with a magnetic component) satisfy the non-attenuated tensor tomography problem of order $n\leq 1$. We instead need the following theorem.

\begin{theorem}\label{theorem:thermostat-tensor-tomography}
Any Gaussian thermostat $(M, g, \theta)$ with $\mathbb{K} <0$ satisfies the attenuated tensor tomography problem of order $1$.
\end{theorem}

This result is a consequence of the work in \cite{assylbekov21}. Their argument relies heavily on the negative thermostat curvature assumption. In particular, most of the heavy lifting is done by \cite[Theorem 3.1]{assylbekov21}, where the \textit{Carleman estimates} for Gaussian thermostats with negative curvature are established (akin to the work in \cite{paternain23}).

\begin{theorem}\label{theorem:carleman}
Let $(M, g, \theta)$ be a Gaussian thermostat with $\mathbb{K}\leq -\kappa$ for some $\kappa>0$. For any integer $m\geq 1$ and parameter $s>0$, we have
$$\sum_{k\geq m} |k|^{2s+1}\Vert u_k\Vert^2\leq \dfrac{1}{\kappa s}\sum_{k\geq m+1} |k|^{2s+1}\Vert (Fu)_k\Vert^2$$
for all $u\in \mathcal{C}^\infty(SM)$. 
\end{theorem}

The rest of the argument is then relatively straightforward for our case, which is less general than the one tackled in \cite{assylbekov21}. We include it here for the sake of completeness, but also to show how it can be simplified.

\begin{proposition}\label{theorem:finite-implies-finite}
Let $(M, g, \theta)$ be a Gaussian thermostat with $\mathbb{K}<0$. Suppose that $f\in \mathcal{C}^\infty(SM)$ has finite Fourier degree and $u\in \mathcal{C}^\infty(SM)$ satisfies $(F+V\lambda)u=f$. Then, the function $u$ also has finite Fourier degree. 
\end{proposition}
\begin{proof}
We follow the argument from \cite[Theorem 5.1]{assylbekov21}. Let $m'\geq 0$ be the Fourier degree of $f$. Since $(F+V\lambda)u=f$, we obtain
$$(Fu)_k=-i\lambda_{1}u_{k-1}+i\lambda_{-1}u_{k+1} \qquad \text{for all } |k|\geq m'+1.$$
As a result, there exists $C>0$ such that
$$\Vert(Fu)_k \Vert^2\leq C(\Vert u_{k-1}\Vert^2+\Vert u_{k+1}\Vert^2) \qquad \text{for all } |k|\geq m'+1.$$
Pick $\kappa>0$ such that $\mathbb{K}\leq -\kappa$, fix $s>eC/\kappa$, and let $m\geq \max(2s+1, m'+1)$. We can apply Theorem \ref{theorem:carleman} to get
\begin{equation*}
\begin{alignedat}{1}
\sum_{|k|\geq m} |k|^{2s+1}\Vert u_k\Vert^2 &\leq \dfrac{C}{\kappa s}\sum_{|k|\geq m+1}|k|^{2s+1}(\Vert u_{k-1}\Vert^2 +\Vert u_{k+1}\Vert^2)\\
&\leq \dfrac{C}{\kappa s}\sum_{|k|\geq m}(|k|+1)^{2s+1} \Vert u_{k}\Vert^2.
\end{alignedat}
\end{equation*}
Since $m\geq 2s+1$, we note that
$$(|k|+1)^{2s+1}=\left(1+\dfrac{1}{|k|}\right)^{2s+1}|k|^{2s+1}\leq \left(1+\dfrac{1}{|k|}\right)^{|k|}|k|^{2s+1}\leq e |k|^{2s+1}\quad \text{for all } |k| \geq m,$$
so that
\begin{equation*}
\begin{alignedat}{1}
\sum_{|k|\geq m} |k|^{2s+1}\Vert u_k\Vert^2 
&\leq \dfrac{eC}{\kappa s}\sum_{|k|\geq m}|k|^{2s+1} \Vert u_{k}\Vert^2.
\end{alignedat}
\end{equation*}
It hence follows that
\begin{equation*}
\begin{alignedat}{1}
\left(1-\dfrac{eC}{\kappa s}\right)\sum_{|k|\geq m} |k|^{2s+1}\Vert u_k\Vert^2 
&\leq 0.
\end{alignedat}
\end{equation*}
However, we have $1-eC/(\kappa s)>0$ by design, so $u_k=0$ for all $|k|\geq m$. 
\end{proof}

\begin{proof}[Proof of Theorem \ref{theorem:thermostat-tensor-tomography}]
By Proposition \ref{theorem:finite-implies-finite}, we know that $u$ is of finite Fourier degree. Suppose, for the sake of contradiction, that $u$ is of degree $k\geq 1$. Then, using the equation $(F+V\lambda)u=f$, we have
$$\left(\eta_+ +\left(1+\dfrac{1}{k}\right)\lambda_1V\right)u_k=(\eta_+ + \lambda_1 V + i\lambda_1)u_k = 0$$
and 
$$\left(\eta_- +\left(1+\dfrac{1}{k}\right)\lambda_{-1}V\right)u_{-k}=(\eta_- + \lambda_{-1} V -i\lambda_{-1})u_{-k} = 0.$$
By \cite[Proposition 6.1]{assylbekov17}, it follows that $u_{\pm k}=0$, which is a contradiction.
\end{proof}

Finally, the proofs of our results rely on the possibility of lifting arbitrary holomorphic $1$-forms to solutions of the transport equation. As explained in Appendix \ref{appendix:A}, where we have relegated most of the work on this front, this is again related to the injectivity of the X-ray transform for thermostats. 

\begin{theorem}\label{theorem:extension-holomorphic}
Let $(M, g, \lambda)$, with $\lambda$ of Fourier degree $\leq 1$, be an Anosov thermostat. For any holomorphic (respectively anti-holomorphic) $1$-form $\tau$ on $M$, there exists a distribution $u\in H^{-1}(SM)$ with $u_k=0$ for all $k\leq 0$ (respectively $k\geq 0$) such that $(F+V\lambda)u=0$ and $u_1 = \pi_1^*\tau$ (respectively $u_{-1}=\pi_{-1}^*\tau$). 
\end{theorem}
\begin{proof}
Let us treat the case where the $1$-form $\tau$ is holomorphic. The anti-holomorphic case is completely analogous. Using Lemma \ref{lemma:raising-lowering-equivalence}, we know that $\pi_1^*\tau \in \Omega_1$ is in the kernel of $\eta_-$. We can hence apply Theorem \ref{theorem:extension-forms}, which tells us that there exists $v\in H^{-1}(SM)$ with $v_0=0$ such that $(F+V\lambda)v=0$ and $v_{-1}+v_1= \pi_1^*\tau$. Since $\pi_1^*\tau \in \Omega_1$, we actually have $v_{-1}=0$. We project the distribution $v$ onto its positive Fourier modes to get $u\coloneqq S_+ v=\sum_{k\geq 1}v_k$. 
For all $k\in \Z$, we then get
$$((F+V\lambda)u)_k=\eta_+u_{k-1}+\eta_-u_{k+1}+ik(\lambda_1 u_{k-1}+\lambda_0 u_k+\lambda_{-1}u_{k+1})=0,$$
which entails that $(F+V\lambda)u=0$. 
\end{proof}

We note that we cannot hope to get such a result for an arbitrary $\lambda\in \mathcal{C}^\infty(SM, \R)$. 

\begin{lemma}\label{lemma:no-go-lemma}
Suppose that $\lambda=\lambda_m+\lambda_{-m}$, where $m\geq 2$ and $\eta_- \lambda_{-m}=0$. Since $\lambda_{-m}$ has isolated zeroes, there exists $a\in \Omega_1$ with $\eta_- a = 0$ and $\lambda_{-m}a\neq 0$. Then, there is no $u\in H^{-1}(SM)$ with $u_k=0$ for all $k\leq 0$ such that $(F+V\lambda)u=0$ and $u_1 = a$. 
\end{lemma}
\begin{proof}
Suppose that such a distribution $u$ exists. For any $k\in \Z$, we must have 
$$0=((F+V\lambda)u)_k = \eta_+ u_{k-1}+\eta_-u_{k+1}+ik(\lambda_m u_{k-m}+\lambda_{-m}u_{k+m}).$$
Therefore, applying this identity to $k=-m+1$, we get $\lambda_{-m}a=\lambda_{-m}u_1=0$, which is a contradiction.
\end{proof}
\section{Action on holomorphic differentials}\label{section:action-holomorphic-differentials}

We have seen that by passing through a specific type of $2$-currents instead of directly using the pullback $\phi^*$, the linear map $\Phi: \mathcal{D}'_{\text{tr}}(SM)\to \mathcal{D}'_{\text{tr}}(\widetilde{S}M)$ defined through the diagram \eqref{eq:definition-phi} sends distributional solutions to the transport equation of one thermostat to those of the second. In this section, we want to show that $\Phi$ can also be seen as acting on holomorphic differentials from one complex surface to another when $\lambda$ is of Fourier degree $\leq 1$ and the attenuated tensor tomography problem of order $1$ is satisfied.
\subsection{Action on fibrewise holomorphic distributions} We start by studying the action of $\Phi$ on the subspace $\mathcal{D}'_{\text{tr}, +}(SM)$ defined in equation \eqref{eq:invariant-fibrewiseholomorphic}. This will require some microlocal analysis. 

We introduce $\mathcal{C}\subset \left\{(v, \xi)\in T^*(SM)\mid \xi(F(v))=0\right\}$, the closed cone enclosed by $E_s^\ast$ and $E_u^\ast$ in the half-space $\left\{(v, \xi)\in T^*(SM)\mid\xi(V(v))\geq 0\right\}$. See Figure \ref{figure:subbundles}. 
\begin{lemma}\label{lemma:wavefrontset}
Let $(M, g, \lambda)$ be an Anosov thermostat. If $u\in \mathcal{D}'_{\textup{tr},+}(SM)$, then we have $\textup{WF}(u)\subset  \mathcal{C}$ and $u_k \in \Omega_k$ for all $k\in \Z$. 
\end{lemma}
\begin{proof}
The argument is essentially the same as that of \cite[Lemma 2.5]{guillarmou25}. Let us give the details. 

By definition, each $u\in \mathcal{D}'_{\textup{tr},+}(SM)$ satisfies $S_+ u = u$. Using the wavefront set description of the Schwartz kernel of $S_+$ (see \cite[Lemma 3.10]{guillarmou17}), we thus get
$$\text{WF}(u)=\text{WF}(S_+ u)\subset \left\{(v, \xi)\in T^*(SM)\mid \xi(V(v))\geq 0\right\}.$$
Given that $(F+V\lambda)u=0$, elliptic regularity tells us that 
$$\text{WF}(u)\subset \Sigma.$$
By propagation of singularities for real principal type differential operators (see \cite[Theorem 26.1.1]{hormander09}), we further know that $\text{WF}(u)$ is invariant by the symplectic lift to $T^*(SM)$ of the flow $\{\varphi_t\}_{t\in \R}$. Given the Anosov property, the maximal flow-invariant subset of $T^*(SM)$ contained in $\Sigma \cap \left\{(v, \xi)\in T^*(SM)\mid \xi(V(v))\geq 0\right\}$ is precisely $\mathcal{C}$, so this gives us the first claim. 

For the second claim, recall that $u_k = (2\pi)^{-1}\pi_k^\ast\pi_{k\ast}u$. 
The pushforward operator $\pi_{k\ast}$ only selects the wavefront set in $(\R X)^\ast\oplus \mathbb{H}^*$ (see \cite[Proposition 11.3.3]{friedlander99}), which is empty given that $\mathcal{C}\cap((\R X)^\ast\oplus \mathbb{H}^*)=\{0\}$ by property \eqref{eq:transversality-cotangent}. Therefore, $u_k\in \Omega_k$.
\end{proof}

We will also need the following lemma with the same proof as \cite[Lemma 3.3]{guillarmou25}.

\begin{lemma}\label{lemma:sub-orientation}
Let $(M, g, \lambda)$ and $(M, \tilde{g}, \tilde{\lambda})$ be Anosov thermostats. Suppose that there exists a smooth orbit equivalence $\phi : \widetilde{S}M\to SM$ isotopic to the identity between them. Then, $\phi$ preserves the natural orientation of the weak unstable subbundle $\R F\oplus E_u$, namely that given by the global frame $(F, Y^u)$, where $Y^u$ is defined in equation \eqref{eq:pre-riccati}.
\end{lemma}

Armed with this, we can show that $\Phi$ maps fibrewise holomorphic distributional solutions to the transport equation of one thermostat to those of the second. For this step of the proof, however, we restrict to thermostats where $\lambda$ has Fourier degree $\leq 1$ and the attenuated tensor tomography problem of order $1$ is satisfied. 

\begin{proposition}\label{proposition:action-1}
Let $(M, g, \lambda)$ and $(M, \tilde{g}, \tilde{\lambda})$, with $\lambda$ and $\tilde{\lambda}$ of Fourier degree $\leq 1$, be Anosov thermostats satisfying the attenuated tensor tomography problem of order $1$. Suppose there exists a smooth orbit equivalence $\phi : \widetilde{S}M\to SM$ isotopic to the identity between them. Then, the map $\Phi$ defined through the diagram \eqref{eq:definition-phi} yields a $\mathbb{C}$-linear isomorphism
$$\Phi:\mathcal{D}'_{\textup{tr},+}(SM)\to 	\mathcal{D}'_{\textup{tr},+}(\widetilde{S}M).$$
\end{proposition}
\begin{proof}
Since $d\phi^\top$ maps connected sets to connected sets, $E_u^*$ to $\widetilde{E}_u^*$, and $E_s^*$ to $\widetilde{E}_s^*$, $d\phi^\top(\mathcal{C})$ must be one of the four cones depicted on the right of Figure \ref{figure:subbundles} (inside the characteristic set $\widetilde{\Sigma}$). It follows that $d\phi^\top(\mathcal{C})=\pm \widetilde{\mathcal{C}}$ because any other cone would entail that $\phi$ reverses the orientation, which is impossible since it is assumed to be isotopic to the identity. If $d\phi^\top(\mathcal{C})=- \widetilde{\mathcal{C}}$, then $\phi$ would flip the orientation of the weak unstable leaves, which contradicts Lemma \ref{lemma:sub-orientation}. Therefore, $d\phi^\top(\mathcal{C})= \widetilde{\mathcal{C}}$.

Let $u\in \mathcal{D}'_{\textup{tr},+}(SM)$ and $\tilde{u}\coloneqq\Phi u$. By Lemma \ref{lemma:wavefrontset}, we know that $\text{WF}(u)\subset \mathcal{C}$. By Lemma \ref{lemma:wavefrontset-pullback}, we thus know that $\text{WF}(\tilde{u})=d\phi^\top (\text{WF}(u))\subset  \widetilde{\mathcal{C}}$.  Then, $\widetilde{S}_- \tilde{u}\in \mathcal{C}^\infty(\widetilde{S}M)$ and, since $(\widetilde{F}+\widetilde{V}\tilde{\lambda})\tilde{u}=0$, we also have that 
\begin{equation*}
\begin{alignedat}{1}
(\widetilde{F}+\widetilde{V}\tilde{\lambda})\widetilde{S}_-\tilde{u}&=[\widetilde{F} + \widetilde{V}\tilde{\lambda}, \widetilde{S}_-]\tilde{u}\\
&=(\tilde{\eta}_++\tilde{\lambda}_1\widetilde{V}+i\tilde{\lambda}_1)\tilde{u}_0-(\tilde{\eta}_-+\tilde{\lambda}_{-1}\widetilde{V}-i\tilde{\lambda}_{-1})\tilde{u}_1.
\end{alignedat}
\end{equation*}

Since $(M, \tilde{g}, \tilde{\lambda})$ satisfies the tensor tomography problem of order $1$ by assumption, it follows that $\widetilde{S}_-\tilde{u}$ is of Fourier degree $0$. Hence, $\tilde{u}$ is fibrewise holomorphic. The fact that $\Phi$ is an isomorphism is then clear as it admits an inverse, namely the map associated to $(\phi^{-1})^*$ by the diagram \eqref{eq:definition-phi}. 
\end{proof}
\subsection{Extension operator} Next, we show how $\Phi$ can be seen as acting on holomorphic differentials from one complex surface to another. Let us start by noting that the map 
\begin{equation}\label{eq:above-extension}
\pi_{1\ast}: \mathcal{D}'_{\text{tr},+}(SM)\to H_J^0(M, \kappa).
\end{equation}
is well-defined. Indeed, if $u \in \mathcal{D}'_{\text{tr},+}(SM)$, then $Xu+V(\lambda u)=(F+V\lambda)u=0$, which means that $(Xu)_0=0$ and hence $\eta_-u_1=0$. By Lemma \ref{lemma:raising-lowering-equivalence}, this is equivalent to $\overline{\partial}\pi_{1\ast} u = 0$. 

Thanks to Theorem \ref{theorem:extension-holomorphic}, we know that the map \eqref{eq:above-extension} is surjective. We can thus define a right-inverse
$$e_1:H_J^0(M, \kappa) \to  \mathcal{D}'_{\text{tr},+}(SM)$$
such that $\pi_{1\ast}\circ e_1= \text{id}_{H_J^0(M, \kappa)}$. We call it an \textit{extension operator}. We may then define the map 
\begin{equation*}
\Psi : H^0_{J}(M, \kappa)\to H^0_{\widetilde{J}}(M, \kappa)
\end{equation*}
by the following commutative diagram.
\begin{equation}
\begin{tikzcd}
\mathcal{D}'_{\text{tr},+}(SM)\arrow[r, "\Phi"] & \mathcal{D}'_{\text{tr},+}(\widetilde{S}M)\arrow[d, "\pi_{1\ast}"]\\
H^0_{J}(M, \kappa)\arrow[u, "e_1"]\arrow[r, "\Psi", dashrightarrow] & H^0_{\widetilde{J}}(M, \kappa)
\end{tikzcd}
\end{equation}
\subsection{Period preservation} The following result shows that the induced mapping of holomorphic differentials we have just defined preserves additional structure. 
\begin{proposition}\label{proposition-last-isomorphism} Let $(M, g, \lambda)$ and $(M, \tilde{g}, \tilde{\lambda})$, with $\lambda$ and $\tilde{\lambda}$ of Fourier degree $\leq 1$, be Anosov thermostats satisfying the attenuated tensor tomography problem of order $1$. Then, the $\mathbb{C}$-linear map 
\begin{equation*}
\Psi : H^0_{J}(M, \kappa)\to H^0_{\widetilde{J}}(M, \kappa)
\end{equation*}
is an isomorphism. It preserves periods in the sense that, for all $[\gamma]\in H_1(M;\Z)$ and $\tau \in H^0_{J}(M, \kappa)$, we have 
$$\int_{[\gamma]}\tau=\int_{[\gamma]} \Psi \tau.$$
\end{proposition}

Recall that there is a push-forward map $\pi_\ast: \mathcal{C}^\infty(SM, \Omega^2(SM))\to \mathcal{C}^\infty(M, \Omega^1(M))$ given by integration along fibres. It satisfies $d\pi_\ast = \pi_\ast d$ (see \cite[Proposition 6.14]{bott82}), and it extends to currents. By \cite[Proposition 6.15]{bott82}, we have the projection formula
\begin{equation}\label{eq:projection-formula}
\int_{\pi^{-1}(\gamma)}\sigma=\int_{\gamma}\pi_\ast \sigma
\end{equation}
for any smooth oriented curve $\gamma$ on $M$ and any $2$-form $\sigma$ on $SM$. 

\begin{lemma}\label{lemma:first-pairing}
Let $(M, g, \lambda)$ be a thermostat and let $\omega$ be the $2$-form on $SM$ defined in equation \eqref{eq:definition-special-form}. For any $u\in \mathcal{D}'(SM)$, we have 
\begin{equation*}
\pi_\ast (u\omega)=\dfrac{1}{2}\star\left(\pi_{-1\ast}u+\pi_{1\ast}u\right).
\end{equation*}
\end{lemma}
\begin{proof} It suffices to establish the claim for $u\in \mathcal{C}^\infty(SM)$. Recall that $\omega\coloneqq\iota_F\mu$. A quick computation then yields
\begin{equation}\label{eq:useful-for-later}
\omega =\beta\wedge\psi + \lambda\alpha\wedge\beta.
\end{equation}
Note that  $\pi^*\mu_a = \alpha\wedge\beta$, where $\mu_a$ is the area form on $M$.

Pick $x\in M$ and take $w\in S_xM$. Then, by definition, we have
$$\pi_\ast (u\omega)_x (w)=\int_{S_x M}\iota_{\tilde{w}}(u\omega),$$
where $\tilde{w}\in T(SM)$ is a lift of $w$ under $d\pi$. We take $\tilde{w}=(w,0)$, that is, no component in the vertical subbundle $\mathbb{V}=\ker d\pi$. Then, since $\psi(\tilde{w})=0$, we get
$$\pi_\ast (u\omega)_x (w)=\int_{S_x M}\iota_{\tilde{w}}(\beta)u\psi.$$
Given that $\iota_{\tilde{w}}\beta_{(x,v)}=g_x( w, Jv)$, we obtain
\begin{equation*}
\begin{alignedat}{1}
\pi_\ast(u\omega)_x(Jw)&=\int_{S_xM} g_x(w, \cdot) u(x,\cdot) \psi\\
&=\int_0^{2\pi} \cos (t) u(\rho_t(x, w))\, dt\\
&=\dfrac{1}{2}\int_0^{2\pi} (e^{it}+e^{-it}) u(\rho_t(x, w))\, dt\\
&=\pi(u_{-1}+u_1)(x, w),
\end{alignedat}
\end{equation*} 
where in the last equality, we used equation \eqref{eq:definition-fourier-mode}. In terms of $1$-forms, since we have $u_{k}=(2\pi)^{-1}\pi_{k}^*\pi_{k\ast}u $, we proved that 
$$-\star \pi_\ast (u\omega)=\dfrac{1}{2}\left(\pi_{-1\ast}u+\pi_{1\ast}u\right).$$
We conclude by applying $\star$ to both sides. 
\end{proof}

We can then integrate this identity, applied to solutions of the transport equation, over closed thermostat geodesics to obtain the following result.

\begin{lemma}\label{lemma:pairing-two}
Let $(M, g, \lambda)$ be an Anosov thermostat and let $\gamma$ be a closed thermostat geodesic. For any $u\in \mathcal{D}'_{\textup{tr}}(SM)$, the pairing $\langle \pi^{-1}(\gamma), u\omega\rangle$ is well defined and
\begin{equation*}
\int_{\pi^{-1}(\gamma)} u \omega =\dfrac{1}{2} \int_{[\gamma]} \star(\pi_{-1\ast}u+\pi_{1\ast}u).
\end{equation*}
\end{lemma}
\begin{proof}
By the wavefront set calculus, the pairing $\langle \pi^{-1}(\gamma), u\omega\rangle$ is well defined whenever
\begin{equation}\label{eq:wavefrontset-condition}
N^*(\pi^{-1}(\gamma))\cap \textup{WF}(u)=\emptyset
\end{equation}
(see \cite[Corollary 8.2.7]{hormander03}  for instance). Since $\mathbb{V}\subset T(\pi^{-1}(\gamma))$, the conormal bundle $N^*(\pi^{-1}(\gamma))$ consists of a line contained in $(\R X)^*\oplus \mathbb{H}^*$. Lemma \ref{lemma:wavefrontset} and property \eqref{eq:transversality-cotangent} then tell us that the intersection with $\text{WF}(u)$ is indeed empty. It follows that the pairing $\langle \pi^{-1}(\gamma), u\omega\rangle$ is well defined and extends the pairing computed for $u\in \mathcal{C}^\infty(SM)$. 

We can then apply the projection formula \eqref{eq:projection-formula} and Lemma \ref{lemma:first-pairing}. By Lemma \ref{lemma:current-isomorphism}, the $2$-current $u\omega$ is closed if $u\in \mathcal{D}'_{\text{tr}}(SM)$, so $\star(\pi_{-1\ast}u+\pi_{1\ast}u)$ is also closed given that $\pi_* d = d\pi_*$, which implies that its integral only depends on the homology class $[\gamma]$. 
\end{proof}

As $\ell_1\star = -V\ell_1$, for any $u\in \mathcal{D}'(SM)$, we may write 
$$\star\left(\pi_{-1\ast}u+\pi_{1\ast}u\right)=i\left(\pi_{-1\ast}u-\pi_{1\ast}u\right).$$
Therefore, if $\tau \in H^0_J(M, \kappa)$, $[\gamma]\in H_1(M; \Z)$ and $\gamma$ is any thermostat geodesic whose homology class is $[\gamma]$, Lemma \ref{lemma:pairing-two} gives us
\begin{equation}\label{eq:translated-to-i}
2i\int_{\pi^{-1}(\gamma)} e_1(\tau)\omega=  \int_{[\gamma]}\tau.
\end{equation}

We can now tackle the proof of Proposition \ref{proposition-last-isomorphism}. 

\begin{proof}[Proof of Proposition \ref{proposition-last-isomorphism}]
Let $[\gamma]\in H_1(M; \Z)$ and let $\gamma$, $\tilde{\gamma}$ be two thermostat geodesics with respect to $(M, g, \lambda)$ and $(M, \tilde{g}, \tilde{\lambda})$, respectively, whose homology class is $[\gamma]$. Since $\phi$ is isotopic to the identity, we know that $[\phi(\pi^{-1}(\tilde{\gamma}))]=[\pi^{-1}(\gamma)]$ in $H_2(SM; \Z)$. 

We claim that the pairing $\langle \phi(\pi^{-1}(\tilde{\gamma})), e_1(\tau)\omega\rangle$ is well defined. In light of condition \eqref{eq:wavefrontset-condition} and Lemma \ref{lemma:wavefrontset}, which tell us that $\text{WF}(e_1\tau)\subset \mathcal{C}$, it is enough to show that
\begin{equation}\label{eq:first-cone-property}
N^*(\phi(\pi^{-1}(\tilde{\gamma})))\cap \mathcal{C}=\{0\}.
\end{equation}
We have seen in the proof of Proposition \ref{proposition:action-1} that $d\phi^\top(\mathcal{C})=\widetilde{\mathcal{C}}$, so combining this observation with the property
$$N^*(\phi(\pi^{-1}(\tilde{\gamma})))=(d\phi^{\top})^{-1}(N^*(\pi^{-1}(\tilde{\gamma})))$$
allows us to rewrite condition \eqref{eq:first-cone-property} as
\begin{equation*}
N^*(\pi^{-1}(\tilde{\gamma}))\cap \widetilde{\mathcal{C}}=\{0\}.
\end{equation*}
Again, since $\widetilde{\mathbb{V}}\subset T(\widetilde{S}M)$, we know that the conormal bundle $N^*(\pi^{-1}(\tilde{\gamma}))$ consists of a line in $(\R \widetilde{X})^*\oplus \widetilde{\mathbb{H}}^*$, so it trivially intersects the closed cone $\widetilde{\mathcal{C}}$, as desired. 

The $2$-current $e_1(\tau)\omega$ is closed by Lemma \ref{lemma:current-isomorphism}. By the Hodge decomposition theorem, we may hence write $e_1(\tau)\omega = \sigma+df$ for some harmonic $2$-current $\sigma$ and $1$-current $f$ with $\text{WF}(f)=\text{WF}(e_1\tau)$. Thanks to the wavefront set condition, the same argument as for $e_1(\tau)\omega$ then shows that both pairings $\langle \pi^{-1}(\gamma), df\rangle$ and $\langle \phi(\pi^{-1}(\tilde{\gamma})), df\rangle$ are well defined. They must be equal to $0$ since $df$ is exact. We thus get
$$\int_{\pi^{-1}(\gamma)}e_1(\tau)\omega = \int_{\pi^{-1}(\gamma)}\sigma =\int_{\phi(\pi^{-1}(\tilde{\gamma}))}\sigma = \int_{\phi(\pi^{-1}(\tilde{\gamma}))}e_1(\tau)\omega,$$
where in the second equality, we have used the facts that $\sigma$ is harmonic and that $[\pi^{-1}(\gamma)]=[\phi(\pi^{-1}(\tilde{\gamma}))]$ in $H_2(SM; \Z)$. 

We can now use equation \eqref{eq:translated-to-i} and unravel the definitions to obtain
\begin{align*}
\int_{[\gamma]}\tau &=2i\int_{\pi^{-1}(\gamma)} e_1(\tau)\omega\\
&=2i\int_{\phi(\pi^{-1}(\tilde{\gamma}))} e_1(\tau)\omega=2i\int_{\pi^{-1}(\tilde{\gamma})} \phi^*(e_1(\tau)\omega)\\
&=2i\int_{\pi^{-1}(\tilde{\gamma})} \Phi(e_1\tau)\tilde{\omega}=\int_{[\gamma]}\pi_{1\ast}\Phi(e_1\tau)=\int_{[\gamma]}\Psi\tau. \qedhere
\end{align*}
\end{proof}

\section{End of the proofs}\label{section:end-of-the-proofs}
\subsection{Torelli's theorem}
The work from the previous section, when combined with Torelli's theorem, tells us that a smooth orbit equivalence isotopic to the identity determines the class $[J]$ in the moduli space $\mathcal{M}(M)$ of complex structures on $M$. 

\begin{proposition}\label{proposition-moduli-space}
Let $(M, g, \lambda)$ and $(M, \tilde{g}, \tilde{\lambda})$, with $\lambda$ and $\tilde{\lambda}$ of Fourier degree $\leq 1$, be Anosov thermostats satisfying the attenuated tensor tomography problem of order $1$. If there exists a smooth orbit equivalence isotopic to the identity between them, then $[J]=[\widetilde{J}]$ in $\mathcal{M}(M)$. Equivalently, there exists a diffeomorphism $\psi : M\to M$ such that $\psi^*\widetilde{J}=J$ and $\psi^*\tilde{g}=e^{2f}g$ for some $f\in \mathcal{C}^\infty(M, \R)$. 
\end{proposition}
\begin{proof}
By Proposition \ref{proposition-last-isomorphism}, the map $\Psi: H^0_J(M, \kappa)\to H^0_{\widetilde{J}}(M, \kappa)$ is a period-preserving $\mathbb{C}$-linear isomorphism. This means that $(M, J)$ and $(M, \widetilde{J})$ have the same period matrix. Indeed, given a canonical basis $\{a_j, b_j\}$ of the homology $H_1(M;\Z)$ on $M$, let $\{\zeta_j\}$ be a basis for $H_J^0(M,\kappa)$ such that property \eqref{eq:basis-condition} is satisfied. Then, $\{\Psi\zeta_j\}$ is a basis for $H^0_{\widetilde{J}}(M, \kappa)$ such that property \eqref{eq:basis-condition} is also satisfied, and
$$(\Pi(J))_{jk}=\int_{b_j}\zeta_k =\int_{b_j}\Psi\zeta_k = (\Pi(\widetilde{J}))_{jk}.$$
Since the surface $M$ must be of genus $\geq 2$, Theorem \ref{theorem:torelli} tells us that there exists an orientation-preserving diffeomorphism $\psi : M\to M$ such that $\psi^*J= \widetilde{J}$. 
\end{proof}

In this section, we want to show something stronger, namely, that the class of the complex structure $J$ is determined in Teichmüller space $\mathcal{T}(M)$.

\begin{proposition}\label{proposition:same-teichmuller-class}
Let $(M, g, \lambda)$ and $(M, \tilde{g}, \tilde{\lambda})$ be either:
\begin{enumerate}
\item[(a)] two Anosov magnetic systems; or
\item[(b)] two Gaussian thermostats with $\mathbb{K}, \widetilde{\mathbb{K}}<0$.
\end{enumerate}
If there exists a smooth orbit equivalence isotopic to the identity between them, then $[J]=[\widetilde{J}]$ in $\mathcal{T}(M)$. Equivalently, there exists a diffeomorphism $\psi : M\to M$ isotopic to the identity such that $\psi^*\widetilde{J}=J$ and $\psi^*\tilde{g}=e^{2f}g$ for some $f\in \mathcal{C}^\infty(M,\R)$. 
\end{proposition}

The reason we need to specify the nature of the two thermostats, as opposed to Proposition \ref{proposition-moduli-space} which deals with general Anosov thermostats of Fourier degree $\leq 1$, is that our proof relies on the following technical lemma (see \cite[Lemma 3.8]{guillarmou25}).
\begin{lemma}\label{lemma-finite-cover}
Let $J$ and $\widetilde{J}$ be two complex structures on $M$ compatible with orientation such that $[J]\neq[\widetilde{J}]$ in $\mathcal{T}(M)$. Then, there exists a finite cover $M'\to M$ such that the lifts $[J'],[\widetilde{J}']\in \mathcal{T}(M')$ are not in the same $\textup{MCG}(M')$-orbit. 
\end{lemma}

Indeed, when lifting thermostats to finite covers, the Anosov property and negative thermostat curvature are preserved; however, satisfying the attenuated tensor tomography problem of order $1$ is not preserved \textit{a priori}.

\begin{proof}[Proof Proposition \ref{proposition:same-teichmuller-class}] Suppose, for the sake of contradiction, that $[J]\neq [\widetilde{J}]$ in $\mathcal{T}(M)$. By Lemma \ref{lemma-finite-cover}, there exists a finite cover $p:M'\to M$ such that the lifts $[J']$ and $[\widetilde{J}']$ are not in the same $\text{MCG}(M')$-orbit. 

Since the smooth orbit equivalence between the flows of $(M, g, \lambda)$ and $(M, \tilde{g}, \tilde{\lambda})$ is isotopic to the identity, it can be lifted to a smooth orbit equivalence isotopic to the identity between the flows of $(M, p^*g, \lambda\circ dp)$ and $(M, p^*\tilde{g}, \tilde{\lambda}\circ dp)$. 

In case (a), we know that the lifted Anosov magnetic flows on $M'$ are again Anosov because the cover is finite. They hence satisfy the (attenuated) tensor tomography problem of order $1$. In case (b), we know that the lifted Gaussian thermostats also have negative thermostat curvature since the property is local. By Theorem \ref{theorem:thermostat-tensor-tomography}, we thus conclude that they satisfy the attenuated tensor tomography problem of order $1$. 

We can then apply Proposition \ref{proposition-moduli-space} to obtain a contradiction.
\end{proof}

Proposition \ref{proposition:same-teichmuller-class} gives us most of Theorems \ref{theorem:second-theorem} and \ref{theorem:main-theorem}. All that is left is studying the rigidity of the function $\lambda$ in each case. 

\subsection{Rigidity of the magnetic field}  
Since $\lambda\in \mathcal{C}^\infty(SM,\R)$ does not depend on velocity in the magnetic case, we will think of it as living in $\mathcal{C}^\infty(M,\R)$. 

\begin{lemma}
Let $(M, g, \lambda)$ be an Anosov magnetic system. Then, the $2$-form $\omega$ on $SM$ defined in equation \eqref{eq:definition-special-form} is exact.
\end{lemma}
\begin{proof}
By equation \eqref{eq:useful-for-later}, we have
$$\omega = -d\alpha + \pi^*(\lambda\mu_a).$$
By the Gauss--Bonnet theorem and the fact $M$ must be of genus $\geq 2$, we know that
$$\int_{M}K_g \mu_a = 2\pi \chi(M)< 0,$$
so $[K_g \mu_a]\neq 0$ in $H^2(M; \R)\cong \R$. It follows that we may write $\lambda\mu_a =c K_g \mu_a + d\varrho$ for some $1$-form $\varrho$ on $M$ and constant $c\in \R$. The constant $c$ is explicitly given by
\begin{equation}\label{eq:constant-c}
c=\dfrac{1}{2\pi \chi(M)}\int_M \lambda\mu_a.
\end{equation} 
However then, since $d\psi = - \pi^*(K_g\mu_a)$, we obtain
\begin{equation*}
\pi^*(\lambda  \mu_a)=c \pi^*(K_g \mu_a) + d\pi^*\varrho=d(-c\psi +\pi^*\varrho),
\end{equation*}
which allows us to write $\omega = d\tau$ for the 1-form
\begin{equation}\label{eq:primitive}
\tau \coloneqq-\alpha-c\psi +\pi^*\varrho,
\end{equation}
as desired.
\end{proof}

Knowing how to find primitives of $\omega$ in the magnetic case then unlocks the following (compare to \cite[Lemma 4.1]{paternain06}).

\begin{proposition}\label{proposition:magnetic-katok}
Let $(M, g, \lambda)$ and $(M, g, \tilde{\lambda})$ be two Anosov magnetic systems with the same background metric $g$. Suppose that there is a smooth conjugacy $\phi : SM\to SM$ isotopic to the identity between them. Then, we have $[\tilde{\lambda}\mu_a]=\pm [\lambda\mu_a]$ in $H^2(M; \R)$. Moreover, $\lambda=0$ if and only if $\tilde{\lambda}=0$.
\end{proposition}
\begin{proof}
Define $\tau$ as in equation \eqref{eq:primitive} to be a primitive of $\omega$. Contracting it with $F$ yields
\begin{equation}\label{eq:tau-with-F}
\tau(F)=-1-c\pi^*\lambda +\ell_1\varrho.
\end{equation}
Recall that the restriction map $\ell_1$ is defined in equation \eqref{eq:defn-ell_1}. Moreover, we know that the Anosov magnetic flows are transitive and $\phi$ is isotopic to the identity, so $\phi^*\mu=\mu$. Since $\phi_*\widetilde{F}=F$, contracting $\phi^*\mu=\mu$ with $\widetilde{F}$ yields $\phi^*\omega=\tilde{\omega}$. This can be rewritten as $\phi^*(d\tau)=d\tilde{\tau}$, which in turn implies $$d(\phi^*\tau-\tilde{\tau})=0.$$ 
Since $\pi^*:H^1(M;\R)\to H^1(SM;\R)$ is an isomorphism, there exist a closed $1$-form $\varphi$ on $M$ and a function $f\in \mathcal{C}^\infty(SM, \R)$ such that
$$\phi^*\tau-\tilde{\tau} = \pi^*\varphi +df.$$
Contracting with $\widetilde{F}$ yields
$$\tau(F)\circ \phi = \tilde{\tau}(\widetilde{F})+\ell_1\varphi +\widetilde{F} f.$$
By equation \eqref{eq:tau-with-F}, we thus get
$$-1+(-c\pi^*\lambda +\ell_1\varrho)\circ \phi=-1-\tilde{c}\pi^*\tilde{\lambda}+\ell_1\tilde{\varrho}+\ell_1\varphi+\widetilde{F} f,$$
which simplifies to
\begin{equation}\label{eq:almost-there}
(-c\pi^*\lambda +\ell_1\varrho)\circ \phi=-\tilde{c}\pi^*\tilde{\lambda}+\ell_1\tilde{\varrho}+\ell_1\varphi+\widetilde{F} f.
\end{equation}
If we integrate with respect to $\mu$, we obtain (since $\phi^*\mu = \mu$)
$$c\int_{SM}\pi^*(\lambda) \mu=\tilde{c}\int_{SM}\pi^*(\tilde{\lambda})\mu.$$
We thus have $c^2=\tilde{c}^2$ by equation \eqref{eq:constant-c}, which shows that the cohomology classes of $\lambda\mu_a$ and $\tilde{\lambda}\mu_a$ in $H^2(M;\R)$ are the same up to a sign. 

If $\lambda=0$, we may take $\varrho=0$, and we know that $c=0$ thanks to equation \eqref{eq:constant-c}. It follows that $\tilde{c}=0$. Let $\tilde{\gamma}$ be a closed magnetic geodesic for $(M, g, \tilde{\lambda})$ of period $T$ and let $\widetilde{\Gamma}\coloneqq(\tilde{\gamma}, \dot{\tilde{\gamma}})\subset SM$. Equation \eqref{eq:almost-there} allows us to write 
\begin{equation*}
\int_0^T\ell_1(\tilde{\varrho}+\varphi)(\widetilde{\Gamma}(t))\, dt=0.
\end{equation*}
By the smooth Livšic theorem \cite[Theorem 2.1]{de-la-llave86} and \cite[Theorem B]{dairbekov05}, this means that $\tilde{\varrho}+\varphi$ is exact. Since $\varphi$ is closed, we get $d\tilde{\varrho}=0$, which in turn implies that $\tilde{\lambda}=0$, as desired. 
\end{proof}

We can now conclude the proof of Theorem \ref{theorem:second-theorem}.

\begin{proof}[Proof of Theorem \ref{theorem:second-theorem}]
If two Anosov magnetic systems $(M, g, \lambda)$ and $(M, \tilde{g}, \tilde{\lambda})$ are related by a smooth conjugacy  $\phi :\widetilde{S}M\to SM$ isotopic to the identity, Proposition \ref{proposition:same-teichmuller-class} yields a diffeomorphism $\psi:M\to M$ isotopic to the identity such that $\psi^*\tilde{g}=e^{2f}g$ for some $f\in \mathcal{C}^\infty(M, \R)$. 

If $\phi$ is a conjugacy and $f=0$, the map $\phi \circ d\psi : SM\to SM$ gives us a smooth conjugacy isotopic to the identity between the Anosov magnetic flows $(M, g, \psi^* \tilde{\lambda})$ and $(M, g, \lambda)$. Thus, Proposition \ref{proposition:magnetic-katok} tells us that $[\psi^*(\tilde{\lambda})\mu_a]=\pm [\lambda\mu_a]$ in $H^2(M;\R)$ and that $\lambda=0$ if and only if $\tilde{\lambda}=0$. 
\end{proof}

\subsection{Rigidity of the thermostat 1-form}  Given the conclusion of Proposition \ref{proposition:same-teichmuller-class}, it behooves us to understand the behaviour of thermostat flows under a conformal rescaling of the metric. 

\begin{lemma}\label{lemma:rescaling}
Let $(M, g, \lambda)$ be a thermostat and define a conformal rescaling $\tilde{g}\coloneqq e^{-2f}g$ of the metric for some $f\in \mathcal{C}^\infty(M,\R)$. Then, the scaling map $s: SM\to \widetilde{S}M$ defined in equation \eqref{eq:scaling-map}, which in this case is simply $(x,v)\mapsto (x, e^fv)$, satisfies
$$s_*X=e^{-f}(\widetilde{X}-\tilde{\ell}_1(\star df)\widetilde{V}), \qquad s_*V= \widetilde{V}.$$
In particular, the map $s$ represents a smooth orbit equivalence isotopic to the identity from the thermostat $(M, g, \lambda)$ to the thermostat
\begin{equation}\label{eq:new-thermostat}
(M, \tilde{g}, e^f(\lambda\circ s^{-1})-\tilde{\ell}_1(\star df)),
\end{equation}
with a time-change  $s_*F = e^{-f}\widetilde{F}$. 
\end{lemma}
\begin{proof} The first statement is proved in \cite[Lemma B.1]{cekic25b}. The conclusion then follows from the calculation 
\begin{align*}
s_*F &=s_*X+(\lambda\circ s^{-1}) s_*V\\
&=e^{-f}(\widetilde{X}-\tilde{\ell}_1(\star df)\widetilde{V})+(\lambda\circ s^{-1})\widetilde{V}.\qedhere\\
\end{align*}
\end{proof}

In what follows, let $\theta$ be the $1$-form on $M$ satisfying $\ell_1\theta = \lambda_{-1}+\lambda_1$. If $\lambda$ is of Fourier degree $\leq 1$, we may more succinctly write the thermostat \eqref{eq:new-thermostat} as
$$(M, e^{-2f}g, e^f \lambda_0 + \tilde{\ell}_1(\theta-\star df)).$$

\begin{proposition}\label{proposition:gaussian-katok}
Let $(M, g, \lambda)$ and $(M, \tilde{g}, \tilde{\lambda})$, with $\lambda$ and $\tilde{\lambda}$ of Fourier degree $\leq 1$, be two Anosov thermostats. Suppose there is a smooth orbit equivalence $\phi : \widetilde{S}M\to SM$ isotopic to the identity between them. If $\star \theta$ or $\tilde{\star}\tilde{\theta}$ is closed, then $\star \theta - \tilde{\star} \tilde{\theta}$ is exact.
\end{proposition}
\begin{proof}
By Lemma \ref{lemma:divergence}, we have $\div_\mu F =V\lambda=- \ell_1(\star \theta)$, so an application of Lemma \ref{lemma:relation-divergences} gives us
$$\phi^*(c)\phi^*( \ell_1(\star\theta))=\tilde{\ell}_1(\tilde{\star}\tilde{\theta})-\widetilde{F}\left(\ln \left(\dfrac{\det \phi}{\phi^*c}\right)\right),$$
where $c\in \mathcal{C}^\infty(SM, \R_{>0})$ is such that $\phi_* \widetilde{F}=cF$. Therefore, if $\tilde{\gamma}$ is a closed thermostat geodesic for $(M, \tilde{g}, \tilde{\theta})$ o period $T$ and $\widetilde{\Gamma} \coloneqq(\tilde{\gamma}, \dot{\tilde{\gamma}})\subset \widetilde{S}M$, we have
$$\int_0^T \phi^*(c\ell_1(\star\theta))(\widetilde{\Gamma}(t))\, dt = \int_{0}^T \tilde{\ell}_1(\tilde{\star}\tilde{\theta})(\widetilde{\Gamma}(t))\, dt.$$
Without loss of generality, suppose that $d(\star \theta)=0$. Then, since $\phi$ is isotopic to the identity and integrals of $1$-forms over curves are independent of the parametrization, we have
$$ \int_0^T \phi^*(c\ell_1(\star\theta))(\widetilde{\Gamma}(t))\, dt = \int_0^T \tilde{\ell}_1(\star\theta)(\widetilde{\Gamma}(t))\, dt .$$
Putting these together, we conclude that
$$\int_{0}^T\tilde{\ell}_1(\star \theta - \tilde{\star}\tilde{\theta})(\widetilde{\Gamma}(t))\, dt=0.$$
An application of the smooth Livšic theorem \cite[Theorem 2.1]{de-la-llave86} and \cite[Theorem B]{dairbekov07} allows us to conclude that $\star \theta - \tilde{\star}\tilde{\theta}$ is exact, as desired. 
\end{proof}

We can now conclude the proof of Theorem \ref{theorem:main-theorem}.

\begin{proof}[Proof of Theorem \ref{theorem:main-theorem}]
If two Gaussian thermostats $(M, g, \theta)$ and $(M, \tilde{g}, \tilde{\theta})$ with negative thermostat curvature are related by a smooth orbit equivalence isotopic to the identity, then Proposition \ref{proposition:same-teichmuller-class} tells us that there exists a diffeomorphism $\psi : M\to M$ isotopic to the identity such that $\psi^*\tilde{g}=e^{2f}g$ for some $f\in \mathcal{C}^\infty(M, \R)$. 

It remains to show that, if either $\star\theta$ or $\tilde{\star}\tilde{\theta}$ is closed, then $\star(\psi^*\tilde{\theta}- \theta)$ is exact.  The map $\phi \circ d\psi$ yields a smooth orbit equivalence isotopic to the identity between the Anosov Gaussian thermostats $(M, e^{2f}g, \psi^*\tilde{\theta})$ and $(M, g, \theta)$. By Lemma \ref{lemma:rescaling}, we may assume without loss of generality that $f=0$. Proposition \ref{proposition:gaussian-katok} then allows us to conclude.
\end{proof}

\appendix \section{Solutions to the transport equation as extensions}\label{appendix:A}

A key ingredient that we use in this paper is Theorem \ref{theorem:extension-holomorphic}: it allows us to extend any holomorphic $1$-form $\tau$ on $M$ (seen as a function on $SM$) to a fibrewise holomorphic distribution $u\in \mathcal{D}'(SM)$ satisfying the transport equation $(F+V\lambda)u=0$. We say that the distribution $u$ is an extension of $\tau$ in the sense that $u_1 = \pi_1^*\tau$. 

This can be seen as part of a larger theme, which is to find distributional solutions of the transport equation $(F+V\lambda)u=0$ with some prescribed Fourier modes. The problem is closely related to the study of the surjectivity of the adjoint of the X-ray transform for thermostats, which in turn is key to understanding the injectivity of the X-ray transform operator itself. 

As an example, the following extension result generalizes \cite[Theorem 1.4]{paternain14}, \cite[Theorem 1.6]{ainsworth15} and \cite[Theorem 1.7]{assylbekov17}, which cover the cases of geodesic flows, magnetic systems and Gaussian thermostats, respectively.
\begin{theorem}\label{theorem:extension-functions}
Let $(M, g, \lambda)$ be an Anosov thermostat. For any $f\in \mathcal{C}^\infty(M)$, there exists $u\in H^{-1}(SM)$ such that $(F+V\lambda)u=0$ and $u_0 = \pi^*f$. 
\end{theorem}

Their arguments crucially rely on the \textit{Pestov identity} (see \cite[Theorem 3.3]{dairbekov07}).

\begin{theorem}
Let $(M, g, \lambda)$ be a thermostat. Then, we have
\begin{equation}\label{eq:pestov-identity}
\Vert VF u\Vert^2=\Vert FVu\Vert^2-\langle \mathbb{K} Vu, Vu\rangle+\Vert Fu\Vert^2
\end{equation}
for all $u\in \mathcal{C}^\infty(SM)$.
\end{theorem}
Recall that the thermostat curvature $\mathbb{K}$ is defined in equation \eqref{eq:general-thermostat-curvature}. In all three papers, the proofs also introduce the following property.

\begin{definition}
Let $\alpha\in [0,1]$. A thermostat $(M, g, \lambda)$ is \textit{$\alpha$-controlled} if 
\begin{equation*}
\Vert F u\Vert^2-\langle\mathbb{K}u, u\rangle\geq \alpha\Vert F u \Vert^2
\end{equation*}
for all $u\in \mathcal{C}^\infty(SM)$.
\end{definition}

They then show that geodesic flows, magnetic systems and Gaussian thermostats are $\alpha$-controlled for some $\alpha>0$ whenever they are Anosov. However, using the Pestov identity \eqref{eq:pestov-identity} for thermostats, the same proof as in \cite[Theorem 3.1]{assylbekov17} goes through for the more general case. 

\begin{theorem}\label{theorem:alpha-regularity}
Let $(M, g, \lambda)$ be an Anosov thermostat. There exists $\alpha>0$ such that
\begin{equation*}
\Vert F u \Vert^2-\langle\mathbb{K} u, u\rangle\geq \alpha\left(\Vert Fu\Vert^2+\Vert u \Vert^2\right)
\end{equation*}
for all $u \in \mathcal{C}^\infty(SM)$. In particular, $(M, g,\lambda)$ is $\alpha$-controlled.  
\end{theorem}

The rest of the argument in \cite{assylbekov17} can then also be recycled to prove Theorem \ref{theorem:extension-functions}. 

The next theorem is again in the spirit of extending functions with low Fourier degree: it deals with functions induced from 1-forms on $M$. The result requires a technical condition, which is that the $1$-form $\theta$ being considered needs to be \textit{solenoidal} in the sense that $\delta \theta = 0$. Here, $\delta$ is the  co-differential with respect to the metric $g$ acting on $1$-forms, that is, $\delta = -\star d\star $. If we write $\ell_1\theta=a_{-1}+a_1\in \Omega_{-1}\oplus \Omega_1$, then $\delta \theta =0$ if and only if $\eta_+ a_{-1}+\eta_- a_1 = 0$ (see \cite{paternain14}).  

\begin{theorem}\label{theorem:extension-forms}
Let $(M, g, \lambda)$ be an Anosov thermostat. If $a\in \Omega_{-1}\oplus \Omega_1$ satisfies $\eta_+ a_{-1}+\eta_- a_1 = 0$, then there exists $u\in H^{-1}(SM)$ with $u_0=0$ such that $(F+V\lambda)u=0$ and $u_{-1}+u_1 = a$. 
\end{theorem}

This is a generalization of \cite[Theorem 1.5]{paternain14}, \cite[Theorem 1.7]{ainsworth15} and \cite[Theorem 1.8]{assylbekov17}, which again deal with geodesic flows, magnetic systems and Gaussian thermostats, respectively. While the condition $u_0=0$ is not explicitly stated there, it follows directly from their proofs, as they construct solutions of the form $u=VTh+a$ for some $h\in L^2(SM)$, where the projection operator $T$, defined below in equation \eqref{eq:defn-T}, kills the Fourier modes of degree $\leq 1$. 

However, adapting the proofs requires some care. We will need the following lemma.

\begin{lemma}\label{lemma:rapid-decay}
For any $\lambda\in \mathcal{C}^\infty(SM)$, 
the $L^\infty$ norms of its Fourier modes $\{\lambda_k\}_{k\in \Z}$ are rapidly decaying in the sense that, for all $\alpha\in \N$, we have
\begin{equation*}
\sup_{k\in \Z}  \Vert \lambda_{k}\Vert_{L^\infty(SM)}k^\alpha<+\infty.
\end{equation*}
\end{lemma}
\begin{proof}
For any point on $SM$, let $U\subset SM$ be an open neighbourhood admitting smooth coordinates $(x, \theta)\in \R^2\times \mathbb{S}^1$ such that $V= \partial_\theta$. The Sobolev embedding theorem gives us a constant $C>0$ such that
\begin{equation*}
\Vert \lambda_k\Vert_{L^\infty(U)}\leq C\sum_{|\beta|\leq 2}\Vert D^\beta \lambda_k\Vert_{L^2(U)} \qquad \text{for all } k\in \Z.
\end{equation*}
Here, we use the multi-index notation $\beta=(\beta_1, \beta_2, \beta_3)\in \N^3$ and $D^\beta\coloneqq\partial_{x_1}^{\beta_1}\partial_{x_2}^{\beta_2}\partial_\theta^{\beta_3}$. 

Using the explicit formula \eqref{eq:definition-fourier-mode} on $U$, we may write
\begin{equation*}
\begin{alignedat}{1}
\lambda_k(x, \theta)&=\dfrac{1}{2\pi}\int_0^{2\pi} \lambda(x, \theta+t)e^{-ikt} \, dt.\\
\end{alignedat}
\end{equation*}
Therefore, still on the domain $U$, we can see that
$$D^\beta \lambda_k= (D^\beta \lambda)_k.$$
By compactness, we can cover $SM$ with a finite number of such open sets $\{ U_j\}$. Then, we get
\begin{equation*}
\begin{alignedat}{1}
\Vert \lambda_k \Vert_{L^\infty(SM)} &= \max_{j} \Vert \lambda_k \Vert_{L^\infty(U_j)}\\
&\leq \max_{j}  C \sum_{|\beta|\leq 2} \Vert (D^\beta\lambda)_k\Vert_{L^2(U_j)}\\
&\leq C \sum_{|\beta|\leq 2} \Vert (D^\beta\lambda)_k\Vert_{L^2(SM)}.
\end{alignedat}
\end{equation*}
Since the $L^2$ norms of the Fourier modes of a smooth function are rapidly decaying, we obtain the desired result. 
\end{proof}

The following lemma has the same proof as \cite[Lemma 4.1]{assylbekov17}. The argument relies on the Pestov identity and Theorem \ref{theorem:alpha-regularity}.
\begin{lemma}\label{lemma:hanming-lemma}
Let $(M, g, \lambda)$ be an Anosov thermostat. Then, there exists a constant $C>0$ such that
$$\Vert u \Vert_{H^1(SM)}\leq C\Vert VFu\Vert_{L^2(SM)}$$
for all $u\in \bigoplus_{|k|\geq 1}\Omega_k$.
\end{lemma}

Next, we define the projection operator $T:\mathcal{C}^\infty(SM)\to \bigoplus_{|k|\geq 2}\Omega_k$ by
\begin{equation}\label{eq:defn-T}
 Tu\coloneqq \sum_{|k|\geq 2}u_k.
\end{equation}
We also define $Q:\mathcal{C}^\infty(SM)\to \bigoplus_{|k|\geq 2}\Omega_k$ as $Q\coloneqq TVF$. 

\begin{lemma}
Let $(M, g, \lambda)$ be an Anosov thermostat. Then, there exists a constant $C>0$ such that 
\begin{equation*}
\Vert u \Vert_{H^1(SM)}\leq C\Vert Qu \Vert_{L^2(SM)}
\end{equation*}
for all $u\in \bigoplus_{|k|\geq 1}\Omega_k$. 
\end{lemma}
\begin{proof}
In this proof, we will let $C>0$ be a constant which can change from line to line to simplify the notation. 

Let $u\in \bigoplus_{|k|\geq 1}\Omega_k$. From the definition of $Q$, we have
\begin{equation*}
\Vert VF u\Vert^2=\Vert (Fu)_1\Vert^2+\Vert (Fu)_{-1}\Vert^2+\Vert Qu\Vert^2.
\end{equation*}
From Lemma \eqref{lemma:hanming-lemma}, we know that
\begin{equation*}
\Vert u \Vert_{H^1(SM)}\leq C\Vert VF u\Vert_{L^2(SM)},
\end{equation*}
so it remains to show that $\Vert (Fu)_{\pm 1}\Vert \leq C \Vert Q u\Vert$. 

By Theorem \ref{theorem:alpha-regularity} and the Pestov identity \eqref{eq:pestov-identity}, we have
\begin{equation*}
\begin{alignedat}{1}
\Vert VF u\Vert^2 &\geq \Vert Fu\Vert^2+\alpha \Vert FVu \Vert^2+\alpha\Vert Vu\Vert^2.
\end{alignedat}
\end{equation*}
Therefore, we obtain
\begin{equation}\label{eq:first-lower-bound}
\Vert Qu\Vert^2\geq \alpha \Vert Vu\Vert^2\geq \alpha \sum_{k\in \Z}k^2\Vert u_k\Vert^2.
\end{equation}
It also gives us
\begin{equation*}
\begin{alignedat}{1}
\Vert Q u \Vert^2&\geq \alpha \Vert F V u \Vert^2\\
&\geq  \alpha \Vert( F V u)_{1} \Vert^2+  \alpha \Vert( F V u)_{-1} \Vert^2\\
&=\alpha \left\Vert 2i\eta_-u_2-\sum_{k\in \Z}k^2\lambda_{-k+1} u_{k}\right\Vert^2+\alpha\left\Vert -2i \eta_+u_{-2}-\sum_{k\in Z}k^2\lambda_{-k-1} u_{k}\right\Vert^2.
\end{alignedat}
\end{equation*}
Therefore,
$$\left\Vert 2i\eta_- u_{2}- \sum_{k\in \Z}k^2\lambda_{-k+ 1} u_{k}\right\Vert\leq C \Vert Qu\Vert$$
and 
$$\left\Vert 2i \eta_+u_{-2}+\sum_{k\in \Z}k^2\lambda_{-k-1} u_{k}\right\Vert\leq C\Vert Qu\Vert.$$
By the reverse triangle inequality, we get
$$\Vert 2\eta_-u_2\Vert - \left\Vert\sum_{k\in \Z }k^2\lambda_{-k+1}  u_{k}\right\Vert\leq C \Vert Qu\Vert$$
and 
$$\Vert 2 \eta_+u_{-2}\Vert-\left\Vert\sum_{k\in \Z }k^2\lambda_{-k-1} u_{k}\right\Vert\leq C\Vert Qu\Vert.$$
By the triangle inequality, the Cauchy--Schwarz inequality, Lemma \ref{lemma:rapid-decay} and property \eqref{eq:first-lower-bound}, we obtain
\begin{equation*}
\begin{alignedat}{1}
\left\Vert\sum_{k\in \Z}k^2\lambda_{-k\pm 1}u_k \right\Vert & \leq \sum_{k\in \Z}k^2\Vert \lambda_{-k\pm 1}\Vert_{L^\infty(SM)} \Vert u_k\Vert\\
&\leq \left( \sum_{k\in \Z} k^2\Vert \lambda_{-k\pm 1}\Vert_{L^\infty(SM)}^2\right)^{1/2}\left( \sum_{k\in \Z} k^2\Vert u_{k}\Vert^2 \right)^{1/2}\\
&\leq C \left( \sum_{k\in \Z} k^2\Vert u_{k}\Vert^2 \right)^{1/2}\\
&\leq C \Vert Q u \Vert. 
\end{alignedat}
\end{equation*}
This gives us 
$$\Vert \eta_-u_2\Vert \leq C \Vert Qu\Vert\qquad \text{and}\qquad \Vert \eta_+u_{-2}\Vert \leq C \Vert Qu\Vert. $$
The result then follows because
\begin{equation*}
\begin{alignedat}{1}
\Vert (Fu)_1\Vert &= \left\Vert \eta_-u_2 + \sum_{k\in \Z} i k\lambda_{-k+1} u_{k}\right\Vert\\
&\leq \Vert \eta_- u_2\Vert + \sum_{k\in \Z}k\Vert \lambda_{-k+1}\Vert_{L^\infty(SM)} \Vert u_{k}\Vert\\
&\leq C \Vert Qu\Vert
\end{alignedat}
\end{equation*}
and 
\begin{align*}
\Vert (Fu)_{-1}\Vert &= \left\Vert \eta_+u_{-2} + \sum_{k\in \Z} ik\lambda_{-k-1} u_{k}\right\Vert\\
&\leq \Vert \eta_+ u_{-2}\Vert + \sum_{k\in \Z}k\Vert \lambda_{-k-1}\Vert_{L^\infty(SM)}\Vert  u_{k}\Vert\\
&\leq C \Vert Qu\Vert.\qedhere
\end{align*}
\end{proof}

The rest of the proof of Theorem \ref{theorem:extension-forms} then goes exactly as in \cite{assylbekov17}.
\bibliographystyle{alpha}
\bibliography{references}

\begin{thebibliography}{dlLMM86}

\bibitem[Ain15]{ainsworth15}
Gareth Ainsworth.
\newblock {The magnetic ray transform on Anosov surfaces}.
\newblock {\em Discrete and Continuous Dynamical Systems}, 35(5):1801--1816,
  2015.

\bibitem[AR21]{assylbekov21}
Yernat~M. Assylbekov and Franklin~T. Rea.
\newblock {The attenuated ray transforms on Gaussian thermostats with negative
  curvature}.
\newblock {\em Preprint, arXiv:2102.04571}, 2021.

\bibitem[AZ17]{assylbekov17}
Yernat~M. Assylbekov and Hanming Zhou.
\newblock {Invariant distributions and tensor tomography for Gaussian
  thermostats}.
\newblock {\em Communications in Analysis and Geometry}, 25(5):895--926, 2017.

\bibitem[BT82]{bott82}
Raoul Bott and Loring~W. Tu.
\newblock {\em {Differential Forms in Algebraic Topology}}, volume~82 of {\em
  Graduate Texts in Mathematics}.
\newblock Springer, New York, 1982.

\bibitem[CP25]{cekic25b}
Mihajlo Cekić and Gabriel~P. Paternain.
\newblock {Resonant forms at zero for dissipative Anosov flows}.
\newblock {\em (To appear in) Geometry \& Topology}, 2025.

\bibitem[dlLMM86]{de-la-llave86}
Rafael de~la Llave, José~Manuel Marco, and Roberto Moriyón.
\newblock {Canonical perturbation theory of Anosov systems and regularity
  results for the Livsic cohomology equation}.
\newblock {\em Annals of Mathematics}, 123(3):537--611, 1986.

\bibitem[DP05]{dairbekov05}
Nurlan~S. Dairbekov and Gabriel~P. Paternain.
\newblock {Longitudinal KAM-cocycles and action spectra of magnetic flows}.
\newblock {\em Mathematical Research Letters}, 12(5):719--730, 2005.

\bibitem[DP07]{dairbekov07}
Nurlan~S. Dairbekov and Gabriel~P. Paternain.
\newblock {Entropy production in Gaussian thermostats}.
\newblock {\em Communications in Mathematical Physics}, 269(2):533--543, 2007.

\bibitem[DS03]{dairbekov03}
Nurlan~S. Dairbekov and Vladimir~A. Sharafutdinov.
\newblock {Some problems of integral geometry on Anosov manifolds}.
\newblock {\em Ergodic Theory and Dynamical Systems}, 23(1):59--74, 2003.

\bibitem[FK92]{farkas92}
Hershel~M. Farkas and Irwin Kra.
\newblock {\em {Riemann Surfaces}}, volume~71 of {\em Graduate Texts in
  Mathematics}.
\newblock Springer, New York, 2 edition, 1992.

\bibitem[FM12]{farb12}
Benson Farb and Dan Margalit.
\newblock {\em {A Primer on Mapping Class Groups}}, volume~49 of {\em Princeton
  Mathematical Series}.
\newblock Princeton University Press, Princeton, NJ, 2012.

\bibitem[FT99]{friedlander99}
F.~Gerard Friedlander and Mark Toshi.
\newblock {\em {Introduction to the Theory of Distributions}}.
\newblock Cambridge University Press, Cambridge, 2 edition, 1999.

\bibitem[Ghy84]{ghys84}
Étienne Ghys.
\newblock {Flots d'Anosov sur les 3-variétés fibrées en cercles}.
\newblock {\em Ergodic Theory and Dynamical Systems}, 4(1):67--80, 1984.

\bibitem[GK80]{guillemin80}
Victor Guillemin and David Kazhdan.
\newblock {Some inverse spectral results for negatively curved 2-manifolds}.
\newblock {\em Topology}, 19(3):301--312, 1980.

\bibitem[GLP25]{guillarmou25}
Colin Guillarmou, Thibault Lefeuvre, and Gabriel~P. Paternain.
\newblock {Marked length spectrum rigidity for Anosov surfaces}.
\newblock {\em Duke Mathematical Journal}, 174(1):131--157, 2025.

\bibitem[GR23]{gogolev23}
Andrey Gogolev and Federico {Rodríguez Hertz}.
\newblock {Smooth rigidity for $3$-dimensional volume preserving Anosov flows
  and weighted marked length spectrum rigidity}.
\newblock {\em Preprint, arXiv:2210.02295}, 2023.

\bibitem[Gro99]{grognet99}
Stéphane Grognet.
\newblock {Flots magnétiques en courbure négative}.
\newblock {\em Ergodic Theory and Dynamical Systems}, 19(2):413--436, 1999.

\bibitem[Gui17]{guillarmou17}
Colin Guillarmou.
\newblock {Invariant distributions and X-ray transform for Anosov flows}.
\newblock {\em Journal of Differential Geometry}, 105(2):177--208, 2017.

\bibitem[Has94]{hasselblatt94}
Boris Hasselblatt.
\newblock {Regularity of the Anosov splitting and of horospheric foliations}.
\newblock {\em Ergodic Theory and Dynamical Systems}, 14(4):645--666, 1994.

\bibitem[H{\"o}r03]{hormander03}
Lars H{\"o}rmander.
\newblock {\em {The Analysis of Linear Partial Differential Operators I:
  Distribution Theory and Fourier Analysis}}.
\newblock Classics in Mathematics. Springer, Berlin, 2 edition, 2003.

\bibitem[H{\"o}r09]{hormander09}
Lars H{\"o}rmander.
\newblock {\em {The Analysis of Linear Partial Differential Operators IV:
  Fourier Integral Operators}}.
\newblock Classics in Mathematics. Springer, Berlin, 2009.

\bibitem[{Mar}23]{marshall-reber23}
James {Marshall Reber}.
\newblock {Deformative magnetic marked length spectrum}.
\newblock {\em Bulletin of the London Mathematical Society}, 55(6):3077--3096,
  2023.

\bibitem[{Mar}24]{marshall-reber24}
James {Marshall Reber}.
\newblock {Corrigendum: Deformative magnetic marked length spectrum}.
\newblock {\em Bulletin of the London Mathematical Society}, 56(12):3920--3923,
  2024.

\bibitem[MP19]{mettler19}
Thomas Mettler and Gabriel~P. Paternain.
\newblock {Holomorphic differentials, thermostats and Anosov flows}.
\newblock {\em Mathematische Annalen}, 373(1):553--580, 2019.

\bibitem[MP20]{mettler20}
Thomas Mettler and Gabriel~P. Paternain.
\newblock {Convex projective surfaces with compatible Weyl connection are
  hyperbolic}.
\newblock {\em Analysis \& PDE}, 13(4):1073--1097, 2020.

\bibitem[MP22]{mettler22}
Thomas Mettler and Gabriel~P. Paternain.
\newblock {Vortices over Riemann surfaces and dominated splittings}.
\newblock {\em Ergodic Theory and Dynamical Systems}, 42(5):1781--1806, 2022.

\bibitem[Pat06]{paternain06}
Gabriel~P. Paternain.
\newblock {Magnetic rigidity of horocycle flows}.
\newblock {\em Pacific Journal of Mathematics}, 225(2):301--323, 2006.

\bibitem[PS23]{paternain23}
Gabriel~P. Paternain and Mikko Salo.
\newblock {Carleman estimates for geodesic X-ray transforms}.
\newblock {\em Annales Scientifiques de l’École Normale Supérieure},
  56(5):1339--1380, 2023.

\bibitem[PSU14]{paternain14}
Gabriel~P. Paternain, Mikko Salo, and Gunther Uhlmann.
\newblock {Spectral rigidity and invariant distributions on Anosov surfaces}.
\newblock {\em Journal of Differential Geometry}, 98(1):147--181, 2014.

\bibitem[PSU23]{paternain23b}
Gabriel~P. Paternain, Mikko Salo, and Gunther Uhlmann.
\newblock {\em {Geometric Inverse Problems: With Emphasis on Two Dimensions}},
  volume 204 of {\em Cambridge Studies in Advanced Mathematics}.
\newblock Cambridge University Press, Cambridge, 2023.

\bibitem[PW08]{przytycki08}
Piotr Przytycki and Maciej~P. Wojtkowski.
\newblock {Gaussian thermostats as geodesic flows of nonsymmetric linear
  connections}.
\newblock {\em Communications in Mathematical Physics}, 277(3):759--769, 2008.

\bibitem[Woj00]{wojtkowski00b}
Maciej~P. Wojtkowski.
\newblock {W-flows on Weyl manifolds and Gaussian thermostats}.
\newblock {\em Journal de Mathématiques Pures et Appliquées},
  79(10):953--974, 2000.

\end{thebibliography}
\end{document}